\documentclass[10pt]{amsart}
\usepackage{amsmath}
\usepackage{amssymb}
\usepackage{amsfonts}
\usepackage{pxfonts}
\usepackage[OT2,T1]{fontenc}
\usepackage[
textwidth=16.0cm, 
textheight=22.0cm,
hmarginratio=1:1,
vmarginratio=1:1]{geometry}
\usepackage{pifont}
\usepackage[english]{babel}
\usepackage{graphicx}

\newtheorem{thm}{Theorem}[subsection]
\newtheorem{cor}[thm]{Corollary}
\newtheorem{lem}[thm]{Lemma}
\newtheorem{prop}[thm]{Proposition}
\theoremstyle{definition}
\newtheorem{defn}[thm]{Definition}
\newtheorem{conj}[thm]{Conjecture}
\newtheorem{prob}[thm]{Problem}
\theoremstyle{remark}
\newtheorem{rem}[thm]{Remark}

\numberwithin{equation}{subsection}
\numberwithin{figure}{subsection}

\newcommand{\diff}{\mathrm{d}}
\newcommand{\C}{{\mathbb C}}
\newcommand{\R}{{\mathbb R}}
\newcommand{\D}{{\mathbb D}}
\newcommand{\expect}{{\mathbb E}}
\newcommand{\Te}{{\mathbb T}}
\newcommand{\Z}{{\mathbb Z}}

\newcommand{\imag}{\mathrm{i}}
\newcommand{\e}{\mathrm{e}}
\newcommand{\hDelta}{\varDelta}

\newcommand{\pv}{\mathrm{pv}}

\newcommand{\Pop}{{\mathbf P}}
\newcommand{\Piop}{{\boldsymbol\Pi}}

\newcommand{\Ordo}{\mathrm{O}}

\newcommand{\Sop}{\mathbf{S}}

\newcommand\Hop{\mathbf{H}}

\DeclareMathOperator{\re}{Re}
\DeclareMathOperator{\im}{Im}

\begin{document}

%---------------------------------------------------------------------
%Insert here the title, affiliations and abstract:
%
\title{Bloch functions, asymptotic variance, and geometric zero packing}

\author{Haakan Hedenmalm}
\address{
Hedenmalm: Department of Mathematics\\
KTH Royal Institute of Technology\\
S--10044 Stockholm\\
Sweden}

\email{haakanh@math.kth.se}

 
\thanks{The research of the author was supported by Vetenskapsr\aa{}det (VR)}
 
\begin{abstract} 
Motivated by a problem in quasiconformal mapping, we introduce a  
problem in complex analysis, with its roots in the mathematical physics 
of the Bose-Einstein condensates in superconductivity. 
The problem will be referred to as \emph{geometric zero packing}, and is 
somewhat analogous to studying Fekete point configurations.
The associated quantity is a density, denoted  $\rho_\C$
in the planar case, and $\rho_{\mathbb{H}}$ in the case of the hyperbolic plane.
We refer to these densities as \emph{discrepancy densities for planar and
hyperbolic zero packing}, respectively, as they measure the impossibility 
of atomizing the uniform planar and hyperbolic area measures.
The universal asymptotic 
variance $\Sigma^2$ associated with the boundary behavior of conformal 
mappings with quasiconformal extensions of small dilatation 
is related to one of these discrepancy densities: $\Sigma^2=
1-\rho_{\mathbb{H}}$. 
We obtain the estimates 
$3.2\times 10^{-5}<\rho_{\mathbb{H}}\le0.12087$, where 
the upper estimate is derived from the estimate from below on $\Sigma^2$
obtained by Astala, Ivrii, Per\"al\"a,  and Prause, and the estimate from 
below is much more delicate. 
In particular, it follows that $\Sigma^2<1$, which in combination with 
the work of Ivrii shows that the maximal fractal dimension of quasicircles 
conjectured by Astala cannot be reached. 
Moreover, along the way, since the universal quasiconformal integral means 
spectrum has the asymptotics 
$\mathrm{B}(k,t)\sim\frac14\Sigma^2 k^2|t|^2$ for small $t$ and $k$, 
the conjectured formula $\mathrm{B}(k,t)=\frac14k^2|t|^2$ is not true. 
As for the actual numerical values of the discrepancy density $\rho_\C$,
we obtain the estimate from above $\rho_\C\le0.061203\ldots$ by using the
equilateral triangular planar zero packing, where the assertion that
equality should hold can be attributed to Abrikosov. The value of 
$\rho_{\mathbb{H}}$ is expected to be somewhat close to that of $\rho_\C$. 
\end{abstract}

\maketitle

\section{Introduction} 

\subsection{Basic notation}
\label{subsec-1.1}
We write $\R$ for the real line
and $\C$ for the complex plane. Moreover, we write 
$\C_\infty:=\C\cup\{\infty\}$ for the extended complex plane 
(the Riemann sphere). For a complex variable $z=x+\imag y\in\C$, let 
\[
\diff s(z):=\frac{|\diff z|}{2\pi},\qquad
\diff A(z):=\frac{\diff x\diff y}{\pi},
\]
denote the normalized arc length and area measures, as indicated. 
Moreover, we shall write 
\[
\varDelta_z:=\frac{1}{4}\bigg(\frac{\partial^2}{\partial x^2}+
\frac{\partial^2}{\partial y^2}\bigg)
\]
for the normalized Laplacian, and
\[
\partial_z:=\frac{1}{2}\bigg(\frac{\partial}{\partial x}-\imag
\frac{\partial}{\partial y}\bigg),\qquad
\bar\partial_z:=\frac{1}{2}\bigg(\frac{\partial}{\partial x}+\imag
\frac{\partial}{\partial y}\bigg),
\]
for the standard complex derivatives; then $\varDelta$ factors as
$\varDelta_z=\partial_z\bar\partial_z$. Often we will drop the subscript 
for these differential operators when it is obvious from the context with
respect to which variable they apply.
We let $\D$ denote the open unit disk, $\Te:=\partial\D$ the unit circle, 
and $\D_e$ the exterior disk:
\[
\D:=\{z\in\C:\,\,|z|<1\},\qquad \D_e:=\{z\in\C_\infty:\,\,|z|>1\}.
\]
We will find it useful to introduce the sesquilinear forms 
$\langle\cdot,\cdot \rangle_\Te$ and $\langle\cdot,\cdot \rangle_\D$,
as given by
\[
\langle f,g \rangle_\Te:=\int_\Te f(z)\bar g(z)\diff s(z),\qquad
\langle f,g\rangle_\D:=\int_\D f(z)\bar g(z)\diff A(z),
\] 
where, in the first case, $f\bar g\in L^1(\Te)$ is required, and in the
second, we need that  $f\bar g\in L^1(\D)$.
%indication of the differentiation variable $z$. 
At times we use the notation $1_E$ for the characteristic function of a 
subset $E$, which equals $1$ on $E$ and vanishes off $E$.

As for distribution theory, a locally area-summable function $u$ will be 
identified with the distribution acting on a test function $\varphi$
according to
\[
u(\varphi)=\int_\C f\varphi \diff A.
\]
The normalization in the area element $\diff A$ is the reason why, e. g., 
$\hDelta\log|z|$ equals $\frac12$ times the unit point mass at the origin 
(and not $\frac{\pi}{2}$ times as would be the case with the standard
area element).

\subsection{The standard weighted Bergman spaces}
\label{subsec-StandBerg}
For $0<p<+\infty$ and $\alpha\in\R$, we introduce the scale of standard 
weighted Lebesgue spaces $L^p_\alpha(\D)$ of (equivalence classes of) Borel 
measurable functions $f:\D\to\C$ with
\[
\|f\|_{L^p_\alpha(\D)}^p:=\int_\D|f(z)|^p(1-|z|^2)^\alpha\diff A(z)<+\infty.
\]
We say that $f\in A^p_\alpha(\D)$ if and only if $f$ is holomorphic in $\D$
and $f\in L^p_\alpha(\D)$. In this case, we will often write 
$\|\cdot\|_{A^p_\alpha(\D)}$ in place of $\|\cdot\|_{L^p_\alpha(\D)}$. 
The spaces $A^p_\alpha(\D)$ are known as
\emph{the standard weighted Bergman spaces}. For $\alpha=0$, we recover the
Bergman spaces: $A^p_0(\D)=A^p(\D)$. For $\alpha\le-1$, it is easy to see that 
the weighted Bergman space is trivial: $A^p_\alpha(\D)=\{0\}$.
On the other hand, for, e.g., polynomials $f$,  
\[
\lim_{\alpha\to-1^+}(\alpha+1)\|f\|_{L^p_\alpha(\D)}^p=\int_\Te|f|^p\diff s=
\|f\|_{H^p(\D)}^p,
\]
where on the right-hand side appears the Hardy space $H^p(\D)$ norm 
(or quasinorm, if $0<p<1$), given by
\[
\|f\|_{H^p(\D)}^p:=\sup_{0<r<1}\int_\Te|f(r\zeta)|^p\diff s(\zeta)
<+\infty.
\]
This means that in a sense, $H^p(\D)$ appears as the limit of spaces 
$A^p_\alpha(\D)$ as $\alpha\to-1^+$.

\subsection{The Bloch space and the Bloch seminorm}
The \emph{Bloch space} consists of those holomorphic functions 
$g:\D\to\C$ that are subject to the seminorm boundedness condition
\begin{equation}
\|g\|_{\mathcal{B}(\D)}:=\sup_{z\in\D}(1-|z|^2)|g'(z)|<+\infty.
\label{eq-Blochnorm}
\end{equation}
Let $\mathrm{aut}(\D)$ denote the group of sense-preserving M\"obius 
automorphism of $\D$. By direct calculation,
\[
\|g\circ\gamma\|_{\mathcal{B}(\D)}=\|g\|_{\mathcal{B}(\D)},\qquad
\gamma\in\mathrm{aut}(\D),
\]  
which says that the Bloch seminorm is invariant under all M\"obius 
automorphisms of $\D$. The subspace 
\[
\mathcal{B}_0(\D):=\big\{g\in\mathcal{B}(\D):\,\,
\lim_{|z|\to1^-}(1-|z|^2)|g'(z)|=0\big\} 
\]
is called the \emph{little Bloch space}. 
An immediate observation we can make at this point is that provided that 
$g(0)=0$, we have the estimate
\begin{equation*}
|g(z)|\le\|g\|_{\mathcal{B}(\D)}\int_0^{|z|}\frac{\diff t}{1-t^2}
=\frac12\,\|g\|_{\mathcal{B}(\D)}\log\frac{1+|z|}{1-|z|},\qquad z\in\D,
%\label{eq-pointwise1}
\end{equation*}
which is sharp pointwise.

\subsection{The Bergman projection of bounded functions}
\label{subsec-PLinfty}

For $f\in L^1(\D)$, let  
\[
\Pop f(z):=\int_\D\frac{\mu(w)}{(1-z\bar w)^2}\,\diff A(w),\qquad z\in\D,
\]
be its \emph{Bergman projection}. Restricted to $L^2(\D)$, it is the 
orthogonal projection onto the subspace of holomorphic functions. In addition, 
it acts boundedly on $L^p(\D)$ for each $p$ in the interval $1<p<+\infty$ 
(see, e.g., \cite{HKZ}).

By appealing to the Hahn-Banach theorem, we may identify the dual space of 
$A^1(\D)$ isometrically and isomorphically with the space
$\Pop L^\infty(\D)$, with respect to the sesquilinear form 
$\langle\cdot,\cdot\rangle_\D$, provided $\Pop L^\infty(\D)$ is equipped with
the canonical norm 
\[
\|g\|_{\Pop L^\infty(\D)}:=\inf\big\{\|\mu\|_{L^\infty(\D)}:\,\,
\mu\in L^\infty(\D)\,\,\,\text{ and }\,\,\, g=\Pop\mu\big\}.
\]
However, since for $f\in A^1(\D)$ and $g\in\Pop L^\infty(\D)$, it may happen 
that $f\bar g$ fails to be in $L^1(\D)$, the identification via the 
sesquilinear form requires some care. The following calculation shows that 
that $\langle f,g\rangle_\D$ remains meaningful for $f\in A^1(\D)$ and 
$g=\Pop\mu$ with $\mu\in L^\infty(\D)$ ($f_r(z):=f(rz)$ denotes 
the $r$-dilate of $f$):
\begin{equation}
\langle f,g\rangle_\D:=\lim_{r\to1^-}\langle f_r,g\rangle_\D=
\lim_{r\to1^-}\langle f_r,\Pop\mu\rangle_\D=\lim_{r\to1^-}\langle\Pop f_r,
\mu\rangle_\D=\lim_{r\to1^-}\langle f_r,\mu\rangle_\D=\langle f,\mu\rangle_\D.
\label{eq-meaningf1}
\end{equation}
Here, we use the facts that the Bergman projection $\Pop$ is self-adjoint 
on $L^2(\D)$ and preserves $A^2(\D)$, and that we have the norm convergence 
$f_r\to f$ as $r\to1^-$ in the space $A^1(\D)$.

It was shown by Coifman, Rochberg, and Weiss \cite{CRW} 
that as a linear space, $\Pop L^\infty(\D)$ equals the Bloch space 
$\mathcal{B}(\D)$, but actually, the endowed norm differs substantially 
from the seminorm \eqref{eq-Blochnorm}. 
Recently, Per\"al\"a \cite{Per} obtained the rather elementary estimate
\begin{equation}
\|\Pop\mu\|_{\mathcal{B}(\D)}\le\frac{8}{\pi}\|\mu\|_{L^\infty(\D)},\qquad 
\mu\in L^\infty(\D),
\label{eq-Perala}
\end{equation}
and showed that the constant $8/\pi$ is best possible. As for lower bounds
up to a little Bloch function,
the best constant is not known, but it is easy to see that the constant 
$1$ works. In conclusion, trying to understand the space $\Pop L^\infty(\D)$ 
in terms of the Bloch seminorm involves a substantial loss of information.

\subsection{Hyperbolic zero packing and the main result}
We mention briefly the topic of optimal discretization of a given positive 
Riesz mass as the sum of unit point masses. The optimization is over the 
possible locations of the various point masses. While this problem has a
classical flavor, it seems to have never been pursued in the precise context
we now present. 
For $r$ with $0<r<1$ and a polynomial $f$, we consider the function
\[
\Phi_f(z):=\big((1-|z|^2)|f(z)|-1\big)^2,\qquad z\in\D,
\]  
which we call the \emph{hyperbolic discrepancy function}. The function $\Phi_f$ 
cannot vanish on a nonempty open subset, because $\Phi_f(z)=0$ means that 
$|f(z)|=(1-|z|^2)^{-1}$. This  is not possible for holomorphic $f$ as in the 
sense of distribution theory, $\hDelta\log|f|$ is a sum of half unit 
point masses, whereas $\hDelta\log\frac{1}{1-|z|^2}=(1-|z|^2)^{-2}$, which 
is a smooth positive Riesz density. 
We are interested in the quantity
\begin{equation}
\rho_{\mathbb{H}}:=
\liminf_{r\to1^-}\inf_f\frac{\int_{\D(0,r)}\Phi_f(z)\frac{\diff A(z)}{1-|z|^2}}
{\int_{\D(0,r)}\frac{\diff A(z)}{1-|z|^2}}=
\liminf_{r\to1^-}\inf_f\frac{\int_{\D(0,r)}\Phi_f(z)\frac{\diff A(z)}{1-|z|^2}}
{\log\frac{1}{1-r^2}},
\label{eq-rhodef1.11}
\end{equation}
where the infimum runs over all polynomials $f$. The number $\rho_{\mathbb{H}}$,
which obviously is confined to the interval $0\le\rho_{\mathbb{H}}\le1$,
will be referred to as the \emph{minimal discrepancy density for hyperbolic 
zero packing}. It measures how close the function $\Phi_f$ can be to $0$,
on average. There is also a more geometric interpretation (compare with 
Remark \ref{rem-geometry}). A very similar density appeared in the context 
of the plane $\C$ in the work of Abrikosov (see \cite{Abr} and \cite{ABN} 
for a more mathematical treatment) on Bose-Einstein condensates in 
superconductivity.  

In connection with the universal asymptotic variance $\Sigma^2$ defined below,
a variant of the density $\rho_{\mathbb{H}}$ is more appropriate, 
which we denote by $\rho_{\mathbb{H}}^\ast$. We write 
\[
\Phi_f(z,r):=\big((1-|z|^2)|f(z)|-1_{\D(0,r)}(z)\big)^2,\qquad z\in\D,\,\,\,
0<r<1,
\]
so that $\Phi_f(z,r)=\Phi_f(z)$ on $\D(0,r)$ while 
$\Phi_f(z,r)=(1-|z|^2)^2|f(z)|^2$ on the annulus $\D\setminus\D(0,r)$. 
The number $\rho_{\mathbb{H}}^\ast$ is defined by
\begin{equation}
\rho_{\mathbb{H}}^\ast:=
\liminf_{r\to1^-}\inf_f\frac{\int_{\D}\Phi_f(z,r)\frac{\diff A(z)}{1-|z|^2}}
{\int_{\D}1_{\D(0,r)}(z)\frac{\diff A(z)}{1-|z|^2}}=
\liminf_{r\to1^-}\inf_f\frac{\int_{\D}\Phi_f(z,r)\frac{\diff A(z)}{1-|z|^2}}
{\log\frac{1}{1-r^2}},
\label{eq-rhodef1.11.1} 
\end{equation}
and we call it the \emph{minimal discrepancy density for tight hyperbolic 
zero packing}. Clearly, we see that $\rho_{\mathbb{H}}\le\rho_{\mathbb{H}}^\ast$. 
In an earlier version of this paper, it was conjectured that 
$\rho_{\mathbb{H}}^\ast=\rho_{\mathbb{H}}$, and some hints  were offered on how 
one might obtain this result based on the polynomial growth 
$\bar\partial$-techniques which were developed in the paper \cite{AHM} 
by Ameur, Hedenmalm, and Makarov. 
Using the suggested approach, this was obtained recently by Wennman 
\cite{Wenn}, so we now have a theorem.

\begin{thm} {\rm (Wennman)}
It holds that $\rho_{\mathbb{H}}^\ast=\rho_{\mathbb{H}}$.
\label{thm-1.4.1}
\end{thm}

Actually, Wennman's theorem also gives some information regarding how big 
need be the degree of an approximately extremal polynomial. 
As a side remark we mention that if $f_0$ is extremal for the problem
\[
\inf_f\frac{\int_{\D(0,r)}\Phi_f(z)\frac{\diff A(z)}{1-|z|^2}}
{\log\frac{1}{1-r^2}}
\]
a variational argument which compares $f_0$ with $f_0+\epsilon h$ 
(where $h$ is polynomial and $\epsilon\in\C$ tends to $0$) shows that 
the extremal function $f_{0}$ meets
\begin{equation}
(1-r^2)f_0(z)+\frac{r^2}{z^2}\int_0^z\zeta f_0(\zeta)\diff\zeta-
\Pop_r\bigg[\frac{f_{0}}{|f_{0}|}\bigg](z)=
 \Pop_r\bigg[(1-|z|^2)f_{0}(z)-\frac{f_{0}(z)}{|f_{0}(z)|}\bigg](z)=0,
\label{eq-vareq1}
\end{equation}
where $\Pop_r$ denotes the Bergman projection corresponding to the disk 
$\D(0,r)$. 

\begin{rem}
(a) 
The number $\arcsin(\rho_{\mathbb{H}}^{1/2})$ describes the asymptotic minimal 
angle between the two vectors $z\mapsto(1-|z|^2)|f(z)|$ and $1$ along a family
of weighted real Hilbert spaces, as  can be seen from Lemma \ref{lem-basic0} 
below. 

\noindent(b)
To better explain geometric zero packing, we also explain the 
planar case where the expression $\Psi_f(z):=(|f(z)|\e^{-|z|^2}-1)^2$ is 
the \emph{planar discrepancy function}.
We believe that the equilateral triangular lattice has a good chance to be 
extremal for planar zero packing, and we explain later
how to evaluate the planar average of the corresponding $\Psi_f$ as an integral
over a single rhombus (which is the union of two adjacent triangles).

\noindent(c) The hyperbolic zero packing problem considered here belongs to
a more extensive family of problems. Indeed, it is equally natural to 
consider, more generally, for positive $\alpha$ and $\beta$, the hyperbolic 
$(\alpha,\beta)$-discrepancy 
function
$\Phi_{f}^{\langle \alpha,\beta\rangle}(z)=((1-|z|^2)^\alpha|f(z)|^\beta-1)^2$. 
The instance $\alpha=\beta=2$
is related to the possible improvement in the application of the 
Cauchy-Schwarz inequality in \cite{HS} and \cite{HS2}.
\end{rem}

We now present the main result of this paper.

\begin{thm}
The minimal discrepancy density for hyperbolic zero packing enjoys the 
following estimate: 
$3.21\times 10^{-5}<\rho_{\mathbb{H}}\le0.12087$.
\label{thm-mddhzp}
\end{thm}

The proof of this theorem is supplied in Section \ref{sec-appendix}. 
The importance of Theorem \ref{thm-mddhzp} comes from its consequences. 

\begin{thm}
Suppose $g=\Pop\mu$, where $\mu\in L^\infty(\D)$, and if $g_r$ denotes the 
dilate $g_r(\zeta):=g(r\zeta)$, then
\[
\limsup_{r\to1^-}\frac{\int_\Te|g_r|^2\diff s}{\log\frac{1}{1-r^2}}
\le(1-\rho_{\mathbb{H}})\,\|\mu\|_{L^\infty(\D)}^2.
\]
\label{cor-MAIN1}
\end{thm}

In other words, with 
\[
\sigma^2(g):=\limsup_{r\to1^-}\frac{\int_\Te|g_r|^2\diff s}{\log\frac{1}{1-r^2}}
\]
as McMullen's asymptotic variance \cite{McM2}, and
\[
\Sigma^2:=\sup\big\{\sigma^2(g):\,g=\Pop\mu,\,\,\|\mu\|_{L^\infty(\D)}=1\big\}
\]
as the universal asymptotic variance, we have that
\begin{equation}
\Sigma^2\le1-\rho_{\mathbb{H}}.
\label{eq-ineg:Sigmarho}
\end{equation}
In fact, we have equality.

\begin{thm}
We have that $\Sigma^2=1-\rho_{\mathbb{H}}$.
\label{thm-equality}
\end{thm}

In the paper \cite{AIPP} by Astala, Ivrii, Per\"al\"a, and Prause, the estimate 
$\Sigma^2\ge0.87913$ was obtained. As a consequence 
of the inequality \eqref{eq-ineg:Sigmarho}, we obtain that  
$\rho_{\mathbb{H}}\le0.12087$. This is where the estimate 
from above of Theorem \ref{thm-mddhzp} comes from.
This estimate is much smaller than the value $1-\frac{\pi}{4}=0.214\ldots$
which is the expected value of the discrepancy density for an appropriately
tailored Gaussian Analytic Function (see Subsection \ref{subsec-stochhyper}). 

Intuitively, the approximately extremal polynomial $f$ for the definition
\eqref{eq-rhodef1.11} of the discrepancy density $\rho_{\mathbb{H}}$
should have its zeros as hyperbolically equidistributed as possible, with
a prescribed density. Since it stands to reason that we may model these
approximately minimizing polynomials by a single holomorphic function 
$f$ in the disk $\D$, we could try to look for $f$ which is a diffential
of order $1$ (or a character-diffential of the same order $1$), periodic
with respect to a Fuchsian group $\Gamma$ such that $\D/\Gamma$ is a
compact Riemann surface. The most natural choice would be to also ask that
the zeros of $f$ are located along a hyperbolic equilateral triangular lattice. 
For instance, we may compare with the analogous planar case the bound 
achieved by the unilateral triangular lattice is $\rho_\C\le0.061203\ldots$. 
However, the structure of hyperbolic lattices is more rigid than the 
corresponding planar one, and the relevant quantities are harder to evaluate.
  
\begin{rem}
McMullen's notion of asymptotic variance is very much related to Makarov's
modelling of Bloch functions as martingales \cite{Mak1}, \cite{Mak2}, 
\cite{Mak3}. Compare also with Lyons' approach \cite{Lyons} to understand 
Bloch functions as maps from hyperbolic Brownian motion to a planar Brownian 
motion (but for it, the speed of the local variance is variable but at least
bounded) \cite{Lyons}.
\end{rem}

We note in passing that in \cite{Hed2}, the related notion of 
\emph{asymptotic tail variance} was introduced.

\subsection{The quasiconformal integral means spectrum 
and the dimension of quasicircles}

For $0<k<1$, we consider the class ${\boldsymbol{\Sigma}}_k$ 
of normalized $k$-quasiconformal mappings 
$\psi:\C_\infty\to\C_\infty$, where $\C_\infty:=\C\cup\{\infty\}$ is the Riemann 
sphere, which preserve the point at infinity and are conformal in the
exterior disk $\D_e$. The normalization is such that the mapping has 
a convergent Laurent expansion of the form
\[
\psi(\zeta)=\zeta+b_0+b_1\zeta^{-1}+b_2\zeta^{-2}+\cdots,\qquad |\zeta|>1.
\] 
The \emph{integral means spectrum} for the function $h:=\log\psi'$ 
(which is defined in $\D_e$ only) is the function
\[
\beta_h(t):=\limsup_{R\to1^+}
\frac{\log\int_\Te |\e^{t h(R\zeta)}|\diff s(\zeta)}{\log\frac{R^2}{R^2-1}},
\qquad t\in\C.
\] 
The \emph{universal} integral means spectrum is obtained as 
$\mathrm{B}(k,t):=\sup_{\psi}\beta_h(t)$,
where $h=\log\psi'$ and $\psi$ ranges over $\boldsymbol{\Sigma}_k$. 
In \cite{Ivrii}, Ivrii obtains the following asymptotics for 
$\mathrm{B}(k,t)$.

\begin{thm}
{\rm(Ivrii)}
The universal integral means spectrum enjoys the asymptotics
\[
\lim_{k\to0^+}\lim_{t\to0}\frac{\mathrm{B}(k,t)}{k^2|t|^2}=
\frac{\Sigma^{\text{2}}}{4}.
\]
\label{thm-ivrii}
\end{thm} 

Here, $\Sigma^{\text{2}}$ is the universal constant which appears in 
\eqref{eq-ineg:Sigmarho}, so that $\Sigma^{\text{2}}\le 1-\rho_{\mathbb{H}}<1$. 
Hence a combination of Theorems \ref{cor-MAIN1} and \ref{thm-ivrii} 
refutes the general conjecture to the effect that 
$\mathrm{B}(k,t)=\frac14 k^2|t|^2$ for real $t$ with $|t|\le2/k$
\cite{jones}, \cite{PS}. 

We now comment on Ivrii's proof of his theorem.
It is important for the proof that for small $k$, the function 
$\frac{1}{k}\log \psi'$ can be modelled by
$\Sop\mu$ for some $\mu\in L^\infty(\D)$ with $\|\mu\|_{L^\infty(\D)}$, where
$\Sop$ denotes the Beurling transform
\[
\Sop\mu(z)=-\pv\int_{\D}\frac{\mu(w)}{(z-w)^2}\diff A(w).
\]
Moreover, after an inversion of the plane, $\Sop\mu$ essentially becomes 
$\Pop\mu$. While this is standard technology in quasiconformal theory, the 
first important observation Ivrii makes is the ``box lemma'', which says that
for $g=\Pop\mu$ with $\|\mu\|_{L^\infty}\le1$, the control of the right-hand 
side integral in 
\[
\int_\Te |g(r\zeta)|^2\diff s(\zeta)=|g(0)|^2+r^2\int_{\D}|g'(r\zeta)|^2
\log\frac{1}{|\zeta|^2}\diff A(\zeta)
\]
can be localized to a hyperbolic disk of large fixed radius instead. 
This is a kind of weak control of square function type (compare with e.g.
Ba\~nuelos \cite{Ban}), which tells us we are in the right ballpark. 
A clever combination with the Lipschitz property of Bloch functions 
\cite{HKZ} then gives the control from above and below, more or less
simultaneously.

Ivrii actually obtains slightly better control than stated above. In any 
case, he also derives the following dimension expansion
via the Legendre transform formalism connecting the dimension and integral
means spectra (see, e.g., \cite{Mak2}, \cite{Mak3}, and \cite{Pom}, p. 241). 

\begin{cor}
{\rm(Ivrii)}
For any $\epsilon>0$, the maximal Minkowski (or Hausdorff) dimension $D(k)$ 
of a $k$-quasicircle has the asymptotic expansion 
\[
D(k)=1+\Sigma^2k^2+\mathrm{O}(k^{8/3-\epsilon})\quad\text{as}\,\,\,\,k\to0^+.
\]
\label{cor-ivrii}
\end{cor}

Here, a \emph{$k$-quasicircle} is simply the image of the unit circle $\Te$ 
under a $k$-quasiconformal mapping of the Riemann sphere $\C_\infty$.
In particular, Astala's well-known conjecture $D(k)=1+k^2$ is incorrect. 
In fact, Prause made the observation that $D(k)<1+k^2$ holds for every $0<k<1$,
based on a combination of Corollary \ref{cor-ivrii} and the methods developed
by Prause and Smirnov \cite{Sm}, \cite{PS}. It might be conjectured that the 
error term in the corollary, $\mathrm{O}(k^{8/3})$, may be improved to 
$\mathrm{O}(k^{3})$.

\subsection{Structure of the paper}

In Section \ref{sec-identities}, some basic identities are mentioned, which are
based on Green's formula as well an explicit calculation involving 
dilates of harmonic functions.
In Section \ref{sec-dilation}, we explore dilational Carleman reverse 
isoperimetry in a Bergman space setting, which later turns out to be closely
connected with the calculation of the density $\rho_{\mathbb{H}}$. 
In Section \ref{sec-calcasvar}, we begin with a seemingly elementary but 
powerful Hilbert space lemma, and then apply it repeatedly in the proofs 
of Theorems \ref{cor-MAIN1} and \ref{thm-equality}.  
In Section \ref{sec-appendix}, we supply the proof of Theorem 
\ref{thm-mddhzp} by first obtaining a local statement, which is then made 
M\"obius invariant, and finally, the estimate is obtained by integration over
the hyperbolic area measure.  
In Section \ref{sec-imprCS}, we begin the semi-expositary part of the paper, 
where we introduce geometric zero packing in the context of the plane and the
hyperbolic plane. We also explore various relations with Gaussian analytic
functions (GAFs) as well as with certain tilings of the plane and the 
hyperbolic plane, respectively.  In Section \ref{sec-beta}, more general
exponents $\beta$ are considered, and a conjecture is made for planar zero
packing which we attribute to Abrikosov. A relation with the LLL-equation 
is mentioned, which is the Bargmann-Fock analogue of the cubic Szeg\H{o}
equation. Finally, in Section \ref{sec-zeropack:logmono}, it is explained
how to interpret the general zero packing problem for compact Riemann surfaces. 
The solution is expressed in terms of what we have decided to call 
``logarithmic monopoles'' but is more commonly referred to as 
``Green functions'' in the literature. The latter is an abuse of notation 
since classical Green functions are not available on compact surfaces (without
boundary). In addition it is explained how the geometric zero packing problem
differs from the classical Fekete configuration problem, in that the Fekete
problem involves Dirichlet energy, while geometric zero packing instead 
involves Bergman energy (in the limit as $\beta\to0$). Finally, in Section
\ref{sec-Feketegeomzero} a rather speculative connection is drawn using
weighted heat flow to connect geometric zero packing on compact surfaces with
$\beta$-deformed Fekete-type problems.

\subsection{Acknowledgements}
I would like to thank Oleg Ivrii and Aron Wennman for reading carefully 
versions of this manuscript, Peter Zograf for helping out with 
character-modular forms, or, more geometrically, sections of holomorphic 
line bundles on Riemann surfaces. Special thanks are due to Aron Wennman for 
helping out with programming. 
In addition, I would like to thank Kari Astala, Alexander 
Borichev, Michael Benedicks, Douglas Lundholm, Istv\'an Prause, and Aron 
Wennman for several valuable conversations. Finally, I should thank the 
referee for several useful remarks.

\section{Identities for dilates of harmonic functions}
\label{sec-identities}

\subsection{Identities involving dilates of harmonic functions}
The following identity interchanges dilations, and although elementary, it 
is quite important. We write $f_r$ and $g_r$ for the dilates $f_r(z):=f(rz)$
and $g_r(z):=g(rz)$, respectively. 

\begin{lem}
Suppose $f,g:\D\to\C$ are two harmonic functions, which are are 
area-integrable: $f,g\in L^1(\D)$. Then we have that
\[
\langle f_r,g\rangle_\D=\langle f,g_r\rangle_\D,
\qquad 0<r<1.
\] 
\label{lem-basic1}
\end{lem}

This is Lemma 5.1.1 in \cite{Hed2}.
We also need the following identity.

\begin{lem}
Suppose $g,h:\D\to\C$ are functions, where $g$ is holomorphic
and $h$ is harmonic. 
If $g\in L^1(\D)$ and $h$ is the Poisson integral of a function in $L^1(\Te)$,
then we have that
\[
\langle z g_r,h\rangle_\Te=\langle g,(\partial h)_r\rangle_\D,
\]
where we write $z$ for the coordinate function $z(\zeta)=\zeta$. 
\label{lem-basic2}
\end{lem}

This is Lemma 5.2.1 in \cite{Hed2}.

\section{Dilational reverse isoperimetry for a  
Bergman space}
\label{sec-dilation}

\subsection{Carleman's isoperimetric inequality,
dilates, and the $L^1$ Bergman space}

The classical isoperimetric inequality says that the area enclosed by 
a closed loop of length $L$ is at most $L^2/(4\pi)$. Torsten Carleman 
(see \cite{Car}, \cite{Str}) found a nice analytical approach to this fact, 
which gave the estimate
\begin{equation}
\|f\|_{A^2(\D)}\le\|f\|_{H^1(\D)},\qquad f\in H^1(\D). 
\label{eq-isoper1}
\end{equation}
Here, $H^1(\D)$ denotes the classical Hardy space, as in Subsection 
\ref{subsec-StandBerg}. 
As for \eqref{eq-isoper1}, the geometrically relevant case is when $f$ 
is the derivative of the conformal mapping from disk $\D$ to the domain 
enclosed by the loop.
There is of course no converse to the isoperimetrical inequality, since 
for a given enclosed area, the length of the boundary may be infinite.
However, if the boundary curve is regularized by replacing it with a
level curve of the Green function, the reverse problem starts to make
sense. We will not need here the appropriate regularized reverse version of
\eqref{eq-isoper1}, but instead the 
%suitable
analogue where the Hardy space $H^1(\D)$ is replaced by the corresponding 
Bergman space $A^1(\D)$ of area-integrable holomorphic functions.

Since $H^1(\D)$ may be thought of as an appropriate limit of the weighted 
Bergman spaces $A^p_\alpha(\D)$ as $\alpha\to-1^+$, it stands to reason that
Carleman's estimate \eqref{eq-isoper1} might be part of a more general 
estimate comparing the norm in $A^p_\alpha(\D)$ with that of 
$A^{2p}_{\alpha+1}(\D)$. We shall be
interested in obtaining a reverse inequality after dilation, 
with $p=1$ and $\alpha=0$: Is it true that, for some positive constant 
$C_2(r)$,
\begin{equation}
\|f_r\|_{A^1(\D)}\le C_2(r)\|f\|_{A^2_1(\D)},\qquad f\in A^2_1(\D)?
\label{eq-isoper3}
\end{equation}
Here, $f_r(\zeta)=f(r\zeta)$ and $0<r<1$. The question at hand is to obtain 
in explicit form, or at least to estimate from above, the optimal constant 
$C_2(r)$, for $0<r<1$. By the Cauchy-Schwarz inequality, we have that
\begin{multline}
\|f_r\|_{A^1(\D)}=\int_\D|f(r\zeta)|\diff A(\zeta)=\frac{1}{r^2}\int_{\D(0,r)}
|f(z)|\diff A(z)
\\
\le\frac{1}{r^2}\Bigg(\int_{\D(0,r)}\frac{\diff A(z)}{1-|z|^2}
\Bigg)^{1/2}\Bigg(\int_{\D(0,r)}|f(z)|^2(1-|z|^2)\diff A(z)
\Bigg)^{1/2}
\\
=\frac{1}{r^2}\bigg(\log\frac{1}{1-r^2}\bigg)^{1/2}
\Bigg(\int_{\D(0,r)}|f(z)|^2(1-|z|^2)\diff A(z)\Bigg)^{1/2}
\\
\le\frac{1}{r^2}\bigg(\log\frac{1}{1-r^2}\bigg)^{1/2}
\Bigg(\int_{\D}|f(z)|^2(1-|z|^2)\diff A(z)\Bigg)^{1/2}
=\frac{1}{r^2}\bigg(\log\frac{1}{1-r^2}\bigg)^{1/2}
\|f\|_{A^2_1(\D)}.
\label{eq-isoper3.01}
\end{multline}
This immediately shows that the optimal constant in \eqref{eq-isoper3} is 
at most
\begin{equation}
C_2(r)\le \frac{1}{r^2}\bigg(\log\frac{1}{1-r^2}\bigg)^{1/2},\qquad 0<r<1.
\label{eq-isoper3.1}
\end{equation}
We intend to improve this estimate. 

\section{Calculation of asymptotic variance via 
hyperbolic zero packing}
\label{sec-calcasvar}

\subsection{Suboptimality of the Cauchy-Schwarz inequality}

We need to analyze the degree of suboptimality in the Cauchy-Schwarz 
inequality in various situations. To this end, the following lemma is helpful.

\begin{lem}
If $\mathcal{H}$ is an $\R$-linear Hilbert space, the following three 
conditions are equivalent for two given vectors $u,v\in\mathcal{H}$ and a 
real $\theta$ with $0\le\theta\le1$:
\smallskip

\noindent{\rm{(a)}}
$\forall c\in\R:\,\,\|u-cv\|_{\mathcal{H}}\ge\theta\|u\|_{\mathcal{H}}$,  
\smallskip

\noindent{\rm{(b)}}
$\forall c\in\R:\,\,\|u-cv\|_{\mathcal{H}}\ge|c|\theta\|v\|_{\mathcal{H}}$, and 
\smallskip

\noindent{\rm{(c)}} $|\langle u,v\rangle_{\mathcal{H}}|\le(1-\theta^2)^{1/2}
\|u\|_{\mathcal{H}}\|v\|_{\mathcal{H}}$.
\label{lem-basic0}
\end{lem}

\begin{proof}
If $v=0$, all the three conditions are trivially met. 
Next, we assume $v\ne0$. By expanding the square, we find that
\[
\|u-cv\|_{\mathcal{H}}^2=\|u\|_{\mathcal{H}}^2+c^2\|v\|_{\mathcal{H}}^2
-2c\langle u,v\rangle_{\mathcal{H}},
\]
which for $v\ne 0$ attains its minimum for $c=\|v\|_{\mathcal{H}}^{-2}
\langle u,v\rangle_{\mathcal{H}}$: 
\[
\inf_{c\in\R}\|u-cv\|_{\mathcal{H}}^2=\|u\|_{\mathcal{H}}^2-
\frac{\langle u,v\rangle_{\mathcal{H}}^2}{\|v\|_{\mathcal{H}}^2}.
\]
The equivalence of (a) and (c) for $n\ne0$ is immediate from this formula.
As for (b), we note that if introduce the reciprocal constant $c'=1/c$, the
inequality reads $\|c'u-v\|_{\mathcal{H}}\ge\theta \|v\|_{\mathcal{H}}$, which 
is the same as (a) if we switch the roles of $u$ and $v$. Moreover, since 
(c) is preserved under such a switch, the equivalence of (b) and (c) now 
follows from the equivalence of (a) and (c). 
\end{proof}

\subsection{Asymptotic variance via hyperbolic zero packing}

We now explain where the bound asserted in Theorem \ref{cor-MAIN1} comes from.

\begin{proof}[Proof of Theorem \ref{cor-MAIN1}]
We assume $0<r<1$. 
By the definition \eqref{eq-rhodef1.11.1} of $\rho_{\mathbb{H}}^\ast$,  
and by using Wennman's theorem $\rho_{\mathbb{H}}=\rho_{\mathbb{H}}^\ast$, 
we see that there exists a parameter $\epsilon=\epsilon(r)$ with 
$0\le\epsilon\le1$ and $\epsilon(r)\to0$ as $r\to1^-$, such that for every
polynomial $f$,
\begin{equation}
\int_{\D}\Phi_f(z,r)\frac{\diff A(z)}{1-|z|^2}\ge(1-\epsilon)
\rho_{\mathbb{H}}\log\frac{1}{1-r^2}.
\label{eq-approxrho1}
\end{equation}
If $\mathcal{H}=L^2_\alpha(\D)$ with $\alpha=-1$, and $u,v$ are the real-valued
functions $u=1_{\D(0,r)}$ and $v(z)=(1-|z|^2)|f(z)|$ , then
\[
\|u-v\|_{\mathcal{H}}^2=\int_{\D}\Phi_f(z,r)\frac{\diff A(z)}{1-|z|^2},
\quad \|u\|_{\mathcal{H}}^2=\log\frac{1}{1-r^2},
\]
and \eqref{eq-approxrho1} expresses that 
\begin{equation}
\|u-v\|_{\mathcal{H}}^2\ge(1-\epsilon)
\rho_{\mathbb{H}}\|u\|_{\mathcal{H}}^2.
\label{eq-approxrho1.01}
\end{equation}
Since $f$ is an arbitrary polynomial we may freely replace $v$ by $cv$ in
\eqref{eq-approxrho1.01}, provided that $c\ge0$: 
\begin{equation}
\|u-cv\|_{\mathcal{H}}^2\ge(1-\epsilon)
\rho_{\mathbb{H}}\|u\|_{\mathcal{H}}^2,\qquad c\ge0.
\label{eq-approxrho1.02}
\end{equation}
But then the inequality of \eqref{eq-approxrho1.02} holds for all $c\in\R$,
since by the proof of Lemma \ref{lem-basic0} the global minimum over $c$ 
is attained at $c=\|v\|_{\mathcal{H}}^{-2}\langle u,v\rangle_{\mathcal{H}}\ge0$
(the case when $\|v\|_{\mathcal{H}}=0$ is trivial). 
Now, from the equivalence of the conditions (a) and (c) of 
Lemma \ref{lem-basic0}, it follows that \eqref{eq-approxrho1.02} is the same
as having
\begin{equation}
(\langle u,v\rangle_{\mathcal{H}})^2\le(1-(1-\epsilon)\rho_{\mathbb{H}})
\|u\|_{\mathcal{H}}^2\|v\|_{\mathcal{H}}^2.
\label{eq-approxrho1.03}
\end{equation}
When we write this out in terms of the chosen functions $u,v$, we obtain
\begin{equation}
\bigg(\int_{\D(0,r)}|f|\diff A\bigg)^2\le
\big(1-(1-\epsilon)\rho_{\mathbb{H}}\big)
\log\frac{1}{1-r^2}\times\int_{\D}|f(z)|^2(1-|z|^2)\diff A(z),
\label{eq-approxrho2}
\end{equation}
So far, $f$ was assumed to be a polynomial, but by approximation, 
\eqref{eq-approxrho2} holds for any holomorphic $f$ in $\D(0,r)$ such that
the right-hand side integral is finite. Next, we pick a bounded holomorphic 
function $h:\D\to\C$ with $h(0)=0$, and apply Lemma \ref{lem-basic2} combined
with \eqref{eq-meaningf1}:
\[
\langle zg_r,h\rangle_\Te=\langle g,(h')_r\rangle_\D=
\langle\Pop\mu,(h')_r\rangle_\D=\langle\mu,(h')_r\rangle_\D.
\]
It now follows that
\[
|\langle zg_r,h\rangle_\Te|\le\|\mu\|_{L^\infty(\D)}\|(h')_r\|_{A^1(\D)}
=\frac{\|\mu\|_{L^\infty(\D)}}{r^{2}}\int_{\D(0,r)}|h'|\diff A,
\]
and we may combine this with the estimate \eqref{eq-approxrho2}, 
where $f=h'$, and arrive at
\begin{equation}
|\langle zg_r,h\rangle_\Te|^2\le\big(1-(1-\epsilon)\rho_{\mathbb{H}}\big)
\frac{\|\mu\|^2_{L^\infty(\D)}}{r^{4}}\log\frac{1}{1-r^2}\times
\int_{\D}|h'(z)|^2(1-|z|^2)\diff A.
\label{eq-approxrho3}
\end{equation}
By the elementary inequality $1-|z|^2\le\log\frac{1}{|z|^2}$ and the standard
Paley identity (\cite{Gar}, p. 236) for the $H^2$ norm (which is a 
consequence of Green's formula), we know that
\[
\int_{\D}|h'(z)|^2(1-|z|^2)\diff A
\le\int_{\D}|h'(z)|^2\log\frac{1}{|z|^2}\diff A=\int_\Te|h|^2\diff s=
\|h\|_{H^2(\D)}^2,
\]
where in the second last step, we used that $h(0)=0$.
We insert this into \eqref{eq-approxrho3}:
\begin{equation*}
|\langle zg_r,h\rangle_\Te|^2\le\big(1-(1-\epsilon)\rho_{\mathbb{H}}\big)
\frac{\|\mu\|^2_{L^\infty(\D)}\|h\|_{H^2(\D)}^2}{r^{4}}\log\frac{1}{1-r^2}.
\end{equation*}
Finally, we plug in $h:=zg_r$, where $g=\Pop\mu$, which yields
\begin{equation*}
\|g_r\|^2_{H^2(\D)}=\|zg_r\|^2_{H^2(\D)}=
|\langle zg_r,zg_r\rangle_\Te|\le\big(1-(1-\epsilon)\rho_{\mathbb{H}}\big)
\frac{\|\mu\|^2_{L^\infty(\D)}}{r^{4}}\log\frac{1}{1-r^2}.
\end{equation*}
Since $\epsilon=\epsilon(r)\to0$ as $r\to 1^-$, the claimed estimate now 
follows.
\end{proof}

We now turn to the proof of Theorem \ref{thm-equality}. 

\begin{proof}[Proof of Theorem \ref{thm-equality}]
Since by Wennman's theorem $\rho_{\mathbb{H}}^\ast=\rho_{\mathbb{H}}$,
the number $\rho_{\mathbb{H}}$ can be understood as the largest nonnegative 
real such that \eqref{eq-approxrho1} holds. The analysis of the equivalence
of \eqref{eq-approxrho1.01} and \eqref{eq-approxrho1.03} (which is the same 
as \eqref{eq-approxrho2}) shows that
\begin{equation}
r^4\|f_r\|_{A^1(\D)}^2\le
\big(1-(1-\epsilon)\rho_{\mathbb{H}}\big)\|f\|_{A^2_1(\D)}^2
\log\frac{1}{1-r^2},
\label{eq-approxrho2.01}
\end{equation}
for any $f\in A^2_1(\D)$, where $\epsilon=\epsilon(r)\to0$ as $r\to1^-$.
Again, the constant $\rho_{\mathbb{H}}$ is largest possible so that 
\eqref{eq-approxrho2.01} holds.
We now express the estimate \eqref{eq-approxrho2.01} in terms of the 
following operator norm bound for the dilation $\mathbf{D}_r$ given by 
$\mathbf{D}_r f(z)=f_r(z)=f(rz)$:
\begin{equation}
\|\mathbf{D}_r\|_{A^2_{1}(\D)\to A^1(\D)}^2\le
\big(1-(1-\epsilon)\rho_{\mathbb{H}}\big)r^{-4}\log\frac{1}{1-r^2}.
\label{eq-approxrho8.2}
\end{equation}
From the optimality of the constant $\rho_{\mathbb{H}}$ in 
\eqref{eq-approxrho2.01} we see that
\begin{equation}
\limsup_{r\to1^-}
\frac{\|\mathbf{D}_r\|^2_{A^2_{1}(\D)\to A^1(\D)}}{\log\frac{1}{1-r^2}}
=1-\rho_{\mathbb{H}}.
\label{eq-approxrho8.3}
\end{equation}
With respect to $\langle\cdot,\cdot\rangle_\D$, the dual space to the
weighted Bergman space $A^2_1(\D)$ is isometrically $H_*^2(\D)$, which is 
just $H^2(\D)$ but equipped with the equivalent norm 
\[
\|f\|^2_{H^2_*(\D)}:=\|f\|_{H^2(\D)}^2+\|f\|_{A^2(\D)}^2.
\]
With respect to the dual action $\langle\cdot,\cdot\rangle_\D$, 
Lemma \ref{lem-basic1} tells us that $\mathbf{D}_r^*=\mathbf{D}_r$, 
and  we recall that isometrically, the dual space to $A^2_{1}(\D)$ is 
$H^2_*(\D)$ while the dual to $A^1(\D)$ is $\Pop L^\infty(\D)$.
Since by basic functional analysis the norm of an operator and its adjoint
are the same, we get from \eqref{eq-approxrho8.3} that
\begin{equation}
\limsup_{r\to1^-}\frac{\|\mathbf{D}_r\|^2_{\Pop L^\infty(\D)\to H^2_*(\D)}}
{\log\frac{1}{1-r^2}}
=1-\rho_{\mathbb{H}}.
\label{eq-approxrho8.3.1}
\end{equation}
For $\mu\in L^\infty(\D)$ and $g=\Pop\mu$, we observe that $\|g_r\|_{A^2(\D)}\le
\|\mu\|_{L^2(\D)}\le \|\mu\|_{L^\infty(\D)}$, which shows that 
\[
\frac{\|g_r\|_{H^2_*(\D)}^2}{\log\frac{1}{1-r^2}}
=\frac{\|g_r\|_{H^2(\D)}^2+\|g_r\|_{A^2(\D)}^2}{\log\frac{1}{1-r^2}}=
\frac{\|g_r\|_{H^2(\D)}^2+\mathrm{O}(1)}{\log\frac{1}{1-r^2}}=
\frac{\|g_r\|_{H^2(\D)}^2}{\log\frac{1}{1-r^2}}+\mathrm{o}(1)
\]
as $r\to1^-$. It follows from this combined with 
\eqref{eq-approxrho8.3.1} that
\begin{equation}
\limsup_{r\to1^-}\frac{\|\mathbf{D}_r\|^2_{\Pop L^\infty(\D)\to H^2(\D)}}
{\log\frac{1}{1-r^2}}
=1-\rho_{\mathbb{H}}.
\label{eq-approxrho8.3.001}
\end{equation}
Now, the left-hand side expresses a uniform version of the asymptotic variance
$\Sigma^2$. The relation \eqref{eq-approxrho8.3.001} entails the following 
statement. There exists some (sparse) sequence of radii $R_j$ such
that $0<R_j<1$, $R_j\to1$ as $j\to+\infty$, and $R_j$ increases with $j$, 
as well as functions  $\mu_j\in L^\infty(\D)$ with $\|\mu_j\|_{L^\infty(\D)}=1$ 
such that if we put $g_j:=\Pop\mu_j$ we have that
\begin{equation}
\lim_{j\to+\infty}\frac{\|(g_j)_{R_j}\|^2_{H^2(\D)}}
{\log\frac{1}{1-R_j^2}}
=1-\rho_{\mathbb{H}}.
\label{eq-approxrho8.3.002}
\end{equation}
But we need to produce a single function $g$ such that 
\eqref{eq-approxrho8.3.002} holds with $g_j$ replaced by $g$.
How to do this?
We make some preliminary observations. 
For $\mu\in L^\infty(\D)$, let $\mu^{(r)}\in L^\infty(\D)$ denote the function
which equals $\mu^{(r)}(z):=\mu(z/r)$ for $|z|<r$ and $\mu^{(r)}(z)=0$ 
elsewhere. A direct calculation verifies that
\[
\Pop\mu^{(r)}(z)=\int_\D\frac{\mu^{(r)}(w)}{(1-z\bar w)^2}\diff A(w)=
r^2\int_\D\frac{\mu(w)}{(1-rz\bar w)^2}\diff A(w)=r^2\Pop\mu(rz),
\]
so that in particular $\Pop\mu^{(r)}(rz)=r^2\Pop\mu(r^2z)$.
We put $r_j:=R_j^{1/2}$ and build the function $\mu\in L^\infty(\D)$ with
$\|\mu\|_{L^\infty(\D)}=1$ as follows: 
\[
\mu(z)=\mu_j^{(r_j)}(z)=\mu_j(z/r_j)
\quad\text{for}\quad r_{j-1}<|z|<r_j,\quad j=2,3,4,\ldots,
\]
whereas $\mu(z):=\mu_1(z/r_1)$ for $|z|<r_1$. Next, we rewrite $\Pop\mu$ 
in the form (since $r_j^2=R_j$ and $g_j=\Pop\mu_j$)
\begin{equation*}
\Pop\mu(z)=\Pop\mu_j^{(r_j)}(z)+\Pop(\mu-\mu^{(r_j)})(z)
=R_j g_j(r_jz)+\Pop(\mu-\mu_j^{(r_j)})(z),\qquad r_{j-1}<|z|<r_j,\quad
j=2,3,4,\ldots. 
\end{equation*}
By definition, $\mu-\mu_j^{(r_j)}$ vanishes on the annulus $r_{j-1}<|z|<r_j$,
and $\mu-\mu_j^{(r_j)}=\mu$ on the annulus $r_j<|z|<1$. We write 
$\mu^\sharp_{j}:=(\mu-\mu_j^{(r_j)})1_{\D(0,r_{j-1})}$ and 
$\mu^\flat_{j}:=(\mu-\mu_j^{(r_j)})1_{\D\setminus\D(0,r_{j})}$ so that 
$\mu-\mu_j^{(r_j)}=\mu^\sharp_{j}+\mu^\flat_{j}$ holds area-almost everywhere 
on $\D$. With respect to $\mu_j^\sharp$ we apply the elementary estimate
\[
\big|\Pop\mu_j^\sharp(z)\big|=\bigg|\int_{\D(0,r_{j-1})}
\frac{\mu(w)-\mu_j(w/r_j)}{(1-z\bar w)^2}\diff A(w)\bigg|\le
2\int_{\D(0,r_{j-1})}
\frac{\diff A(w)}{|1-z\bar w|^2}\diff A(w)=\frac{2}{|z|^2}
\log\frac{1}{1-R_{j-1}|z|^2}.
\]
With respect to $\mu^\flat_{j}$, on the other hand, we apply the 
alternative elementary estimate
\[
\big|\Pop\mu_j^\flat(z)\big|=\bigg|\int_{\D\setminus\D(0,r_{j})}
\frac{\mu(w)}{(1-z\bar w)^2}\diff A(w)\bigg|\le
\int_{\D\setminus\D(0,r_{j})}
\frac{\diff A(w)}{|1-z\bar w|^2}\diff A(w)=\frac{1}{|z|^2}
\log\frac{1-R_j|z|^2}{1-|z|^2}.
\]
Both estimates need to be applied for $|z|=r_j$:
\[
\big|\Pop\mu_j^\sharp(z)\big|\le\frac{2}{R_j}\log\frac{1}{1-R_{j-1}R_j},\quad
\big|\Pop\mu_j^\flat(z)\big|\le\frac{1}{R_j}\log(1+R_j),\quad\text{if}\,\,\,\,
|z|=r_j.
\]
We now find that
\[
g_{r_j}(z)=g(r_jz)=R_jg_j(R_jz)+\Pop(\mu-\mu_j^{(r_j)})(r_jz),
\] 
where
\[
\big|\Pop(\mu-\mu_j^{(r_j)})(r_jz)\big|\le 
\frac{2}{R_j}\log\frac{1}{1-R_{j-1}R_j}+\frac{1}{R_j}\log(1+R_j),\quad
\text{if}\quad|z|=1.
\]
In particular, since the $H^2(\D)$-norm is dominated by the supremum norm
on the circle $\Te$,
we obtain that
\[
\big\|R_j^{-1}g_{r_j}-(g_j)_{R_j}\big\|_{H^2(\D)}\le 
\frac{2}{R_j^2}\log\frac{1}{1-R_{j-1}R_j}+\frac{1}{R_j^2}\log(1+R_j),
\]
so that
\begin{equation}
\frac{\big\|R_j^{-1}g_{r_j}-(g_j)_{R_j}\big\|_{H^2(\D)}}
{\sqrt{\log\frac{1}{1-R_j^2}}}\le 
\frac{2}{R_j^2}\frac{\log\frac{1}{1-R_{j-1}R_j}}{\sqrt{\log\frac1{1-R_j^2}}}+
\frac{1}{R_j^2}\frac{\log(1+R_j)}{\sqrt{\log\frac{1}{1-R_j^2}}}.
\label{eq-normdiff1.001}
\end{equation}
By passing to a subsequence, we are free to make the sequence of radii 
$R_j$ as sparse as we need. So we just pick them so that the right-hand 
side of \eqref{eq-normdiff1.001} tends to $0$ as $j\to+\infty$. 
Finally, by taking square roots in \eqref{eq-approxrho8.3.002} we know that
\begin{equation*}
\lim_{j\to+\infty}\frac{\|(g_j)_{R_j}\|_{H^2(\D)}}
{\sqrt{\log\frac{1}{1-R_j^2}}}
=(1-\rho_{\mathbb{H}})^{1/2},
%\label{eq-approxrho8.3.002}
\end{equation*}
which together with the estimate \eqref{eq-normdiff1.001} and the sparsity of
the radii leads to 
\begin{equation*}
\lim_{j\to+\infty}\frac{\|(g)_{r_j}\|_{H^2(\D)}}
{\sqrt{\log\frac{1}{1-R_j^2}}}
=(1-\rho_{\mathbb{H}})^{1/2},
%\label{eq-approxrho8.3.002}
\end{equation*}
which entails that
\[
\limsup_{r\to1^-}\frac{\|(g)_{r}\|_{H^2(\D)}}
{\sqrt{\log\frac{1}{1-r^2}}}
\ge(1-\rho_{\mathbb{H}})^{1/2},
\]
since
\[
\lim_{r\to1^-}\frac{\log\frac{1}{1-r^2}}{\log\frac{1}{1-r^4}}=1.
\]
As we have now found a single $\mu$ which does the job, the estimate
$\Sigma^2\ge1-\rho_{\mathbb{H}}$ follows. Since the reverse inequality 
$\Sigma^2\le1-\rho_{\mathbb{H}}$ was obtained in Theorem \ref{cor-MAIN1}, the
equality $\Sigma^2=1-\rho_{\mathbb{H}}$ is immediate.
\end{proof}

\section{The proof of the estimate from below 
on $\rho_{\mathbb{H}}$}
\label{sec-appendix}

\subsection{Pointwise estimates}

For a real parameter $\alpha$, let 
$\diff A_\alpha(z)=(1-|z|^2)^\alpha\diff A(z)$ denote the standard 
weighted area measure on $\D$. We shall need the following estimate.

\begin{lem}
We have the following pointwise estimates, for a holomorphic function 
$f:\D\to\C$ and $0<r<1$:
\[
|f(z)|^2\le
r^{-2}\sum_{j=0}^{+\infty}\frac{(j+1)(j+2)}{1+(j+1)(1-r^2)}
\bigg(\frac{|z|^2}{r^2}\bigg)^{j}
\int_{\D(0,r)}|f|^2\diff A_1,\qquad z\in\D(0,r),
\]
and
\[
|f'(z)|^2\le
r^{-4}\sum_{j=0}^{+\infty}\frac{(j+1)^2(j+2)(j+3)}{1+(j+2)(1-r^2)}
\bigg(\frac{|z|^2}{r^2}\bigg)^{j}
\int_{\D(0,r)}|f|^2\diff A_1,\qquad z\in\D(0,r).
\]
\label{lem-pointwise1.02}
\end{lem}

\begin{proof}
Let $\mathcal{H}$ denote the Hilbert space of holomorphic functions $f$ 
on $\D(0,r)$ that are $L^2$--integrable with respect to the measure 
$\diff A_1$. The Bergman kernel representation of a function 
$f\in\mathcal{H}$ is
\[
f(z)=\langle f,K_r(\cdot,z)\rangle_{\mathcal{H}}
=\int_{\D(0,r)}K_r(z,w)f(w)\diff A_1(w),
\]
where $K_r$ is the corresponding weighted Bergman kernel
\[
K_r(z,w):=\sum_{j=0}^{+\infty}
r^{-2j-2}\frac{(j+1)(j+2)}{1+(j+1)(1-r^2)}(z\bar w)^j.
\]
The corresponding representation for the derivative is
\[
f'(z)=\langle f,\bar\partial_zK_r(\cdot,z)\rangle_{\mathcal{H}}
=\int_{\D(0,r)}\partial_z K_r(z,w)f(w)\diff A_1(w).
\]
Now, by elementary Hilbert space methods, the optimal estimate for the 
value and derivative are, respectively,
\[
|f(z)|^2\le K_r(z,z)\|f\|^2_{\mathcal H},\qquad|f'(z)|^2\le \hDelta_z
K_r(z,z)\|f\|^2_{\mathcal H},
\]
for $z\in\D(0,r)$. Now, as for $K_r(z,z)$ we have that 
\[
K_r(z,z)=\sum_{j=0}^{+\infty}r^{-2j-2}\frac{(j+1)(j+2)}{1+(j+1)(1-r^2)}|z|^{2j},
\qquad z\in\D(0,r).
\]
Finally, as regards $\hDelta_z K_r(z,z)$, we have that 
\[
\hDelta_z K_r(z,z)=
\sum_{j=1}^{+\infty}r^{-2j-2}\frac{j^2(j+1)(j+2)}{1+(j+1)(1-r^2)}|z|^{2j-2},
\qquad z\in\D(0,r).
\]
This completes the proof of the lemma.
\end{proof}

Let $\nabla:=(\partial_x,\partial_y)$ stand for the usual gradient, if 
$z=x+\imag y$ is the representation of the complex coordinate.

\begin{lem}
Suppose $f:\D\to\C$ is holomorphic and nontrivial. Then the function
$z\mapsto (1-|z|^2)|f(z)|$ has local minima only at the zeros of $f$.  
Moreover, for $0<r<1$, the gradient of this function enjoys the estimate
\[
\big|\nabla\big((1-|z|^2)|f(z)|\big)\big|
\le A(r,z)\bigg\{\int_{\D(0,r)}|f|^2\diff A_1\bigg\}^{1/2},
\qquad z\in\D(0,r),
\]
where
\begin{equation*}
A(r,z):=
2|z|K_r(z,z)^{1/2}+(1-|z|^2)(\hDelta_z K(z,z))^{1/2},
\end{equation*}
and $K_r$ is the reproducing kernel of the proof of Lemma 
\ref{lem-pointwise1.02}.
\label{lem-pointwise1.03}
\end{lem}

\begin{proof}
Since 
\[
\hDelta\log\{(1-|z|^2)|f(z)|\}=-(1-|z|^2)^{-2}<0
\] 
holds away from the zeros of $f$, the critical points of the function 
$z\mapsto \log(1-|z|^2)|f(z)|$ can only be local maxima or saddle points,
and this carries over to the function $z\mapsto(1-|z|^2)|f(z)|$ as well.
The estimate of the gradient uses the estimates of 
Lemma \ref{lem-pointwise1.02} together with the product rule
\[
\nabla[(1-|z|^2)|f(z)|]=-|f(z)|\nabla|z|^2+(1-|z|^2)\nabla|f(z)|
\]
and the facts that $|\nabla|z|^2|=2|z|$ and $|\nabla|f||\le|f'|$. 
The estimates of $|f(z)|$ and $|f'(z)|$ are from Lemma \ref{lem-pointwise1.02}.
\end{proof}

\subsection{The fundamental local estimate}

We need to estimate the hyperbolic zero packing constant $\rho_{\mathbb{H}}$
from below.
The hard part consists in obtaining the following local estimate.  

\begin{prop}
There exists an absolute constant $\rho_1$, with $0<\rho_1<1$, such that
for holomorphic $f:\D\to\C$, 
\[
\rho_1
\le 
\int_{\D(0,\frac12)}\big(|f(z)|(1-|z|^2)-
1\big)^2\frac{\diff A(z)}{1-|z|^2}.
\]
For instance, $\rho_1:=2.08\times 10^{-5}$ will do. 
\label{prop-2.1}
\end{prop}

\begin{proof}

We write $f=bf_0$, where $b$ is a positive constant, where $f_0$ is normalized:
\begin{equation}
\int_{\D(0,\frac12)}|f_0(z)|^2(1-|z|^2)\diff A(z)=1.
\label{eq-norm1.01}
\end{equation}
So, we need to show that 
\begin{equation}
\rho_1\le 
\inf_{b>0}\int_{\D(0,\frac12)}\big(b\,|f_0(z)|(1-|z|^2)-
1\big)^2\frac{\diff A(z)}{1-|z|^2}.
\label{eq-reform1.001}
\end{equation}
It is possible to verify that for $r=\frac12$, the function $A(\frac12,z)$
of Lemma \ref{lem-pointwise1.03} is radial and increasing with $|z|$, and 
plugging in $|z|=\frac25$ we obtain from numerical work that
that
\[
A(\tfrac12,z)\le 67,\qquad |z|\le\tfrac25.
\] 
Now, from the normalization \eqref{eq-norm1.01} we get from Lemma 
\ref{lem-pointwise1.03} that
\begin{equation}
\big|\nabla\big((1-|z|^2)|f_0(z)|\big)\big|\le
67,\qquad z\in\D(0,\tfrac25),
\label{eq-deriv1.01}
\end{equation}
where we decided to estimate on a slightly smaller disk.
Next, a straightforward calculus exercise shows that the minimum over $b>0$ is 
attained at the value
\[
b=b_{f_0}:=\int_{\D(0,\frac12)}|f_0(z)|\diff A(z).
\] 
Moreover, a well-known calculation shows that
\begin{equation}
\int_{\D(0,\frac12)}\big(b_{f_0}|f_0(z)|(1-|z|^2)-1\big)^2\frac{\diff A(z)}{1-|z|^2}
=\int_{\D(0,\frac12)}\frac{\diff A(z)}{1-|z|^2}-b_{f_0}^2=\log\frac43-b_{f_0}^2,
\label{eq-p1.2}
\end{equation}
so that in particular, $b_{f_0}^2\le \log\frac43$. Next, we split our argument
according to the size of $b_{f_0}$.
\medskip

\noindent{\sc Case I}. \emph{Suppose that $b_{f_0}^2\le\frac12\log\frac43$}.
Then by \eqref{eq-p1.2}, the claimed estimate holds whenever $\rho_1\le
\frac12\log\frac43$. 

\medskip

\noindent{\sc Case II}. \emph{Suppose that 
$\frac12\log\frac43<b_{f_0}^2\le\log\frac43$}.
We let $F$ be the function $F(z):=(1-|z|^2)|f(z)|=b_{f_0}(1-|z|^2)|f_0(z)|$, 
so that by \eqref{eq-deriv1.01}, we know that
\begin{equation}
|\nabla F(z)|\le 67b_{f_0}\le 36,\qquad w\in \D(0,\tfrac25). 
\label{eq-gradest1.1}
\end{equation}
For a positive real number $\epsilon$, to be specified later, we consider 
the set $\Omega(f,\epsilon)$ given by
\[
\Omega(f,\epsilon):=\big\{z\in\D(0,\tfrac13):\,\,
(1-|z|^2)|f(z)|\ge\epsilon\big\}.
\]
We divide Case II further according to the properties of the set 
$\Omega(f,\epsilon)$.
\medskip

\noindent{\sc Case IIa}: \emph{Suppose that} 
$\Omega(f,\epsilon)\ne\D(0,\frac13)$. 
Then $|F(z_0)|=(1-|z_0|^2)|f(z_0)|<\epsilon$ at some point 
$z_0\in\D(0,\frac13)$, and, in view of \eqref{eq-gradest1.1} and convexity, 
\[
|F(z)|\le\epsilon+36|z-z_0|,\qquad z\in\D(0,\tfrac25).
\]
In particular, if $\epsilon\le\frac{1}{10}$, we see that
\begin{equation}
\int_{\D(0,\frac12)}\big(F(z)-1\big)^2\frac{\diff A(z)}{1-|z|^2}\ge
\int_{\D(z_0,\frac{1}{80})}\big(F(z)-1\big)^2\frac{\diff A(z)}{1-|z|^2}\ge
\bigg(1-\frac{36}{80}-\frac1{10}\bigg)^2\frac{1}{80^2}\ge 3.16\times 10^{-5}.
\label{eq-estbelow1.00101}
\end{equation} 
\medskip

\noindent{\sc Case IIb}: \emph{Suppose that} 
$\Omega(f,\epsilon)=\D(0,\frac13)$. To fit with the argument of {\sc Case IIa}
we should assume that $\epsilon\le\frac1{10}$. In particular, $u:=\log|f|$ is
harmonic, and $F(z)=(1-|z|^2)\e^{u(z)}$. We will make use of the 
following elementary estimate:
\begin{equation}
(1-t)^2\ge \frac{(1-\epsilon)^2}{(\log\frac{1}{\epsilon})^2}
(\log{t})^2,\qquad \epsilon\le t<+\infty,
\label{eq-estbelow1.00102}
\end{equation}
which follows from the monotonicity of the expression 
$\frac{1-\e^{-s}}{s}$ where $s=\log\frac1t$. It is immediate from 
\eqref{eq-estbelow1.00102} and from our assumption 
$\Omega(f,\epsilon)=\D(0,\frac13)$ that
\begin{multline}
\int_{\D(0,\frac12)}\big(F(z)-1\big)^2\frac{\diff A(z)}{1-|z|^2}\ge
\frac{(1-\epsilon)^2}{(\log\frac{1}{\epsilon})^2}
\int_{\D(0,\frac13)}\big(\log F(z)\big)^2\frac{\diff A(z)}{1-|z|^2}
\\
=\frac{(1-\epsilon)^2}{(\log\frac{1}{\epsilon})^2}
\int_{\D(0,\frac13)}\big(\log(1-|z|^2)+u(z)\big)^2\frac{\diff A(z)}{1-|z|^2}
\\
=\frac{(1-\epsilon)^2}{(\log\frac{1}{\epsilon})^2}
\int_{\D(0,\frac13)}\Big(\big(\log(1-|z|^2)\big)^2+2u(z)\log(1-|z|^2)
+u(z)^2\Big)\frac{\diff A(z)}{1-|z|^2}.
\label{eq-estbelow1.00103}
\end{multline}
We calculate that
\[
\int_{\D(0,\frac13)}\big(\log(1-|z|^2)\big)^2\frac{\diff A(z)}{1-|z|^2}=
\frac{(\log\frac98)^3}{3},
\]
and, moreover, by the mean value property of harmonic functions, that
\[
\int_{\D(0,\frac13)}u(z)\log(1-|z|^2)\frac{\diff A(z)}{1-|z|^2}=
u(0)\int_{\D(0,\frac13)}\log(1-|z|^2)\frac{\diff A(z)}{1-|z|^2}=-
\frac{(\log\frac98)^2}{2}u(0).
\]
Furthermore, since $u$ is harmonic, the square $u^2$ is subharmonic, and
hence
\[
\int_{\D(0,\frac13)}u(z)^2\frac{\diff A(z)}{1-|z|^2}\ge
u(0)^2\int_{\D(0,\frac13)}\frac{\diff A(z)}{1-|z|^2}=
u(0)^2\log\frac98.
\]
Adding up the terms, we now obtain from \eqref{eq-estbelow1.00103} that
\begin{multline}
\int_{\D(0,\frac12)}\big(F(z)-1\big)^2\frac{\diff A(z)}{1-|z|^2}\ge
\frac{(1-\epsilon)^2}{(\log\frac{1}{\epsilon})^2}\bigg(\frac{(\log\frac98)^3}{3}
-(\log\tfrac98)^2u(0)+u(0)^2\log\frac98\bigg)
\\
=\frac{(1-\epsilon)^2}{(\log\frac{1}{\epsilon})^2}
\bigg(\Big(u(0)-\frac12\log\frac98\Big)^2\log\frac98
+\frac1{12}\Big(\log\frac98\Big)^3\bigg)\ge
\frac{(1-\epsilon)^2}{12(\log\frac{1}{\epsilon})^2}\Big(\log\frac98\Big)^3.
\label{eq-estbelow1.00104}
\end{multline}
Finally, we specify that $\epsilon:=\frac1{10}$ so that 
\eqref{eq-estbelow1.00104} then gives that
\begin{equation}
\int_{\D(0,\frac12)}\big(F(z)-1\big)^2\frac{\diff A(z)}{1-|z|^2}\ge
2.08\times 10^{-5}.
\label{eq-estbelow1.00105}
\end{equation}

By comparing the estimates we obtained in the cases {\sc I, IIa, IIb}, we 
obtain that the assertion of the propostion holds with 
$\rho_1=2.08\times 10^{-5}$.
\end{proof}

\begin{rem}
We should mention that when asked, Borichev \cite{Bor} came up with an 
absolute lower bound via a somewhat different argument. 
\end{rem}

\subsection{Modification of the fundamental local estimate}

As it turns out, we will need to compare locally
not just with the constant $1$ but with a family of functions whose logarithms 
are harmonic.

\begin{prop}
There exists an absolute constant $\rho_2$ with $0<\rho_2<1$, 
such that for all holomorphic $f:\D\to\C$ and all points $\xi\in\D$, 
\[
\rho_2\le 
\int_{\D(0,\frac12)}\big(|f(z)|(1-|z|^2)-|1-\bar\xi z|^{-1}\big)^2
\frac{\diff A(z)}{1-|z|^2}.
\]
\label{prop-2.2}
For instance, $\rho_2=\frac49\rho_1$ will do, where $\rho_1$ is the 
constant of Proposition \ref{prop-2.1}.
\end{prop}

\begin{proof}
We consider the auxiliary holomorphic function
$g(z):=(1-\bar\xi z)f(z)$. An application of Proposition \ref{prop-2.1} with
$g$ in place of $f$ gives that
\begin{multline*}
\int_{\D(0,\frac12)}\big(|f(z)|(1-|z|^2)-
|1-\bar\xi z|^{-1}\big)^2\frac{\diff A(z)}{1-|z|^2}=
\int_{\D(0,\frac12)}\big(|g(z)|(1-|z|^2)-1\big)^2
|1-\bar\xi z|^{-2}\frac{\diff A(z)}{1-|z|^2}
\\
\ge\frac49\int_{\D(0,\frac12)}\big(|g(z)|(1-|z|^2)-
1\big)^2\frac{\diff A(z)}{1-|z|^2}\ge\tfrac{4}{9}\rho_1,
\end{multline*} 
which expresses the asserted estimate.
\end{proof}

\subsection{The global estimate from below}
We now turn the local estimate into a global one.

\begin{proof}[Proof of Theorem \ref{thm-mddhzp}]
As mentioned in the introduction, the estimate from above 
$\rho_{\mathbb{H}}\le 0.12087$ follows from the work of Astala, Ivrii,
Per\"al\"a, and Prause \cite{AIPP}, 
so it remains to establish the estimate from below.
Our starting point is Proposition \ref{prop-2.2}, which tells us that 
there exists an absolute constant $\rho_2$, with $0<\rho_2<1$,
such that for each $\lambda\in\D$ and each holomorphic function $h:\D\to\C$,
\begin{equation}
\rho_2\le\int_{\D(0,\frac12)}
\big(|h(z)|(1-|z|^2)-|1-\bar\lambda z|^{-1}\big)^2
\frac{\diff A(z)}{1-|z|^2}.
\label{eq-2.1}
\end{equation}
Given $\lambda\in\D$, we introduce the mapping $\gamma_\lambda$ 
given by
\[
\gamma_\lambda(\zeta):=\frac{\lambda-\zeta}{1-\bar\lambda\zeta},
\]
which is an involutive M\"obius automorphism of the unit disk $\D$ 
(so that $\gamma_\lambda\circ\gamma_\lambda(\zeta)=\zeta$).
Moreover, a direct calculation shows that the derivative of $\gamma_\lambda$ 
equals
\[
\gamma_\lambda'(\zeta)=-\frac{1-|\lambda|^2}{(1-\bar\lambda\zeta)^2}.
\]
We make the auxiliary observation that
\begin{equation}
1-|\gamma_\lambda(\zeta)|^2=
\frac{(1-|\lambda|^2)(1-|\zeta|^2)}{|1-\bar\lambda\zeta|^2}=(1-|\zeta|^2)
|\gamma_\lambda'(\zeta)|.
\label{eq-2.2}
\end{equation}
Let $h_\lambda$ denote the holomorphic function
\[
h_\lambda(\zeta):=(-\gamma'_\lambda(\zeta))^{3/2}h\circ\gamma_\lambda(\zeta)=
\frac{(1-|\lambda|^2)^{3/2}}{(1-\bar\lambda\zeta)^3}\,
h\bigg(\frac{\lambda-\zeta}{1-\bar\lambda\zeta}\bigg),
\]
and observe that by \eqref{eq-2.2} and the change-of-variables formula,
\begin{multline}
\int_{\gamma_\lambda(\D(0,\frac12))}\big(|h_\lambda(\zeta)|(1-|\zeta|^2)-1\big)^2
\frac{\diff A(\zeta)}{1-|\zeta|^2}
\\
=\int_{\gamma_\lambda(\D(0,\frac12))}\bigg(|h\circ\gamma_\lambda(\zeta)|
(1-|\gamma_\lambda(\zeta)|^2)-|\gamma_\lambda'(\zeta)|^{-1/2}\big)^2
\frac{|\gamma_\lambda'(\zeta)|^2}{1-|\gamma_\lambda(\zeta)|^2}\diff A(\zeta)
\\
=\int_{\D(0,\frac12)}\big(|h(z)|(1-|z|^2)-|\gamma_\lambda'(z)|^{1/2}
\big)^2\frac{\diff A(z)}{1-|z|^2}
\\
=(1-|\lambda|^2)
\int_{\D(0,\frac12)}\big(|\tilde h(z)|(1-|z|^2)-|1-\bar\lambda z|^{-1}
\big)^2\frac{\diff A(z)}{1-|z|^2}\ge(1-|\lambda|^2)\rho_2,
\label{eq-2.4}
\end{multline}
where $\tilde h(z)=(1-|\lambda|^2)^{-1/2}h(z)$, and, in the last step, 
we invoked \eqref{eq-2.1} with $\tilde h$ in place of $h$. 
If we write $H$ in place of $h_\lambda$, we obtain from 
\eqref{eq-2.4} that 
\begin{equation*}
\int_{\gamma_\lambda(\D(0,\frac12))}\big(|H(\zeta)|(1-|\zeta|^2)-1\big)^2
\frac{\diff A(\zeta)}{1-|\zeta|^2}\ge (1-|\lambda|^2)\rho_2.
%\label{eq-2.4''}
\end{equation*}
This inequality holds in fact for every holomorphic function $H:\D\to\C$,
since for given $H$ it is possible to write down $h$ such that 
$H=h_\lambda$.
We are of course free to integrate both sides with respect to a positive 
finite measure:
\begin{equation}
\int_{\D(0,r^{4})}
\int_{\gamma_\lambda(\D(0,\frac12))}\big(|H(\zeta)|(1-|\zeta|^2)-1\big)^2
\frac{\diff A(\zeta)\diff A(\lambda)}{(1-|\zeta|^2)(1-|\lambda|^2)^2}
\ge \rho_2\int_{\D(0,r^{4})}\frac{\diff A(\lambda)}{1-|\lambda|^2}=\rho_2
\log\frac{1}{1-r^{8}}.
\label{eq-2.4''.1}
\end{equation}
Moreover, we calculate that 
\begin{equation*}
\int_{\D(0,r^{4})}1_{\gamma_\lambda(\D(0,\frac12))}(\zeta)
\frac{\diff A(\lambda)}{(1-|\lambda|^2)^2}
=\int_{\D(0,r^{4})}1_{\D(0,\frac12)}(\gamma_\zeta(\lambda))
\frac{\diff A(\lambda)}{(1-|\lambda|^2)^2}\le 1_{\D(0,r)}(\zeta)\log\frac43,
%\label{eq-2.4''.2}
\end{equation*}
for $r_1<r<1$, provided $r_1<1$ is close enough to $1$, where the bound by
$\log\frac43$ is a consequence of hyperbolic invariance, and the fact
the left hand side vanishes is a consequence of a simple comparison of 
the hyperbolic lengths of the intervals $[0,\frac12]$ and $[r^4,r]$ 
(the latter interval is longer for $r_1<r<1$). It now follows from
\eqref{eq-2.4''.1} that
\begin{equation}
\frac{\rho_2}{\log\frac43}\log\frac{1}{1-r^{8}}
\le\int_{\D(0,r)}\big(|H(\zeta)|(1-|\zeta|^2)-1\big)^2
\frac{\diff A(\zeta)}{1-|\zeta|^2},\qquad r_1<r<1.
\label{eq-2.4''.3}
\end{equation}
Since with $\rho_2=\frac49\rho_1$ and $\rho_1=2.08\times10^{-5}$, we have
the inequality of constants 
\[
\frac{\rho_2}{\log\frac43}>3.21\times10^{-5}.
\]
Moreover, since
\[
\lim_{r\to1^-}\frac{\log\frac{1}{1-r^{8}}}{\log\frac{1}{1-r^{2}}}=1,
\]
the claimed assertion follows from \eqref{eq-2.4''.3}.
\end{proof}

\section{Geometric packing of zeros}
\label{sec-imprCS}

In this section, we develop a rather general type of extremal problems in 
complex analysis, which we call \emph{geometric zero packing problems}. We 
first explain the  planar zero packing problem, and then turn to the 
hyperbolic zero packing problem, which was mentioned earlier. 

\subsection{A packing problem for zeros in the plane}
\label{subsec-packingproblem1}
We first study a packing problem for zeros pertaining to the Bargmann-Fock 
space of entire functions. 
It is well-known that there is no entire function $f:\C\to\C$ such that
$\log|f(z)|=|z|^2$. The reason is that in the sense of distribution theory, 
$\hDelta\log|f|$ is a sum of half unit point masses located at the zeros 
of $f$ (counting multiplicities), so that off the zeros, $\log|f|$ is 
harmonic, while $\hDelta|z|^2=1$. In particular, the nonnegative function 
$(|f(z)|\e^{-|z|^2}-1)^2$ cannot vanish on a nonempty open set, and if
$f(z)$ is a polynomial in $z$, then in particular 
$|f(z)|=\mathrm{O}(\e^{|z|^2})$ as $|z|\to+\infty$, and so we would know 
that the \emph{discrepancy function}
\[
\Psi_f(z):=(|f(z)|\e^{-|z|^2}-1)^2
\] 
is bounded. Note also that for the trivial function $f=0$, the 
discrepancy $\Psi_f=\Psi_0$ equals the constant $1$. It is now a natural 
question to ask \emph{how small the discrepancy $\Psi_f$ can be, on average}, 
since it cannot vanish on nonempty open sets. 
So, we consider the minimal average of $\Psi_f$ in a disk $\D(0,R)$ of large 
radius $R$:
\begin{equation}
\rho_{\C}(R):=\inf_f\frac{1}{R^2}\int_{\D(0,R)}\Psi_f(z)\diff A(z)
=\inf_f\frac{1}{R^2}\int_{\D(0,R)}(|f(z)|\e^{-|z|^2}-1)^2\diff A(z),
\label{eq-packR}
\end{equation}
where the infimum is taken over all polynomials $f$. Here, the use of
the origin as the base point is inessential since in \eqref{eq-packR},
we can take the infimum over all entire $f$ without changing the value of 
$\rho_\C(R)$, and, in addition, by the change-of-variables formula, 
we have for $a\in\C$ the translation invariance property
\[
\frac{1}{R^2}\int_{\D(a,R)}\Psi_f(z)\diff A(z)
=\frac{1}{R^2}\int_{\D(0,R)}\Psi_{f_{\langle a\rangle}}(z)\diff A(z),
\]  
where $f_{\langle a\rangle}$ denotes the Fock-space translate  
$f_{\langle a\rangle}(z):=\e^{-|a|^2-2\bar a z}f(a+z)$. In view of Lemma
\ref{lem-basic0}, this discrepancy density $\rho_\C(R)$ gives the best 
constant for the improved Cauchy-Schwarz inequality
\[
\bigg\{\frac{1}{R^2}\int_{\D(0,R)}|f(z)|\e^{-|z|^2}\diff A(z)\bigg\}^{2}\le
(1-\rho_\C(R))\frac{1}{R^2}\int_{\D(0,R)}|f(z)|^2\e^{-2|z|^2}\diff A(z).
\]

\begin{defn}
For the above problem, the 
the \emph{minimal discrepancy density for planar zero packing} is 
$\rho_{\C}:=\liminf_{R\to+\infty}\rho_\C(R)$.
\label{defn-odp1}
\end{defn}

\begin{rem}
(a) The limsup might be considered as well, but we expect it to 
equal the liminf.

\noindent (b) In more geometric terms, the quantity $\rho_{\C}$ is a measure 
of how well the planar metric $\diff s=|\diff z|$ can be approximated
by a metric obtained in the following manner:
take the surface with the Gaussian metric $\diff s=|f(z)|\e^{-|z|^2}|\diff z|$, 
where $f$ is a polynomial, which then has curvature  
\[
-4\hDelta\log(|f(z)|\e^{-|z|^2})=4-2\sum_{j}\delta_{w_j},
\]
where $\{w_j\}_j$ are the zeros of $f$, and $\delta_\xi$ is the unit mass 
delta function at $\xi\in\C$. The point masses in the curvature correspond
to ``branch'' or ``flabby cone'' points with an opening of $4\pi$ in case
of simple zeros, and more generally, an opening of $2(n+1)\pi$ for a zero 
of multiplicity $n$.
\label{rem-geometry}
\end{rem}

Since polynomials are determined up to a multiplicative constant by their 
zeros, we feel that the terminology ``geometric zero packing'' or ``geometric 
packing of zeros'' is appropriate. 

\begin{prob}
Determine the value of $\rho_\C$. For which 
configurations of zeros of the polynomial $f$ is it asymptotically attained? 
Is the equilateral triangular lattice optimal asymptotically?
\end{prob}

In Conjecture \ref{conj-betaAbrik} below we attribute the conjecture that the
equilateral triangular lattice is optimal (in the more general context of 
an exponent $\beta$) to Abrikosov.  
We illustrate with an equilateral triangular tesselation in Figure 
\ref{fig-1}.

\begin{figure}[h]
\includegraphics
[
%trim = 0cm 15cm 0cm 15cm, clip=true,
%trim = 0cm 8cm 0cm 8cm, clip=true,
%totalheight=0.42\textheight
height=6cm, width=8cm]
{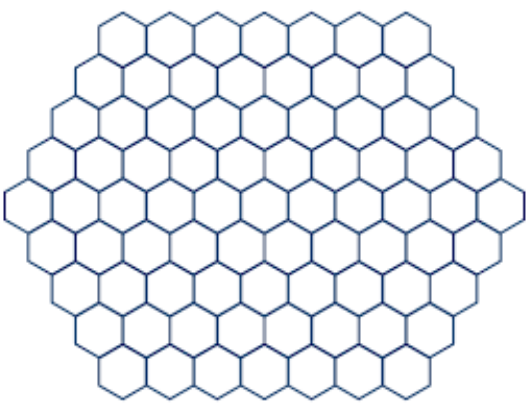}
%{hexagonal.pdf}
\caption{Illustration of the honeycomb lattice, with a zero to be placed at
the center of each hexagon to produce the equilateral triangular tesselation.}
\label{fig-1}
\end{figure}

The Weierstrass sigma function $\sigma(z)$, which arises in the analysis
of the Weierstrass $\wp(z)$ function, can be used to analyze the 
asymptotic discrepancy density for the equilateral triangular lattice
(see, e. g., the exposition of Ahlfors \cite{Ahlfors}).
Let $\omega_1,\omega_2\in\C$ be the periods associated with the lattice
\[
\Lambda_{\omega_1,\omega_2}:=\omega_1\Z+\omega_2\Z,
\]
where it is assumed that $\omega_1,\omega_2$ are $\R$-linearly independent.
For simplicity, we 
suppose $\omega_1$ is real with $\omega_1>0$, and that 
$\im \omega_2>0$.
The related Weierstrass function $\wp(z)$ then has the complex periods 
$\omega_1$ and $\omega_2$.  
We recall the formula for the associated 
sigma function (see, e. g., \cite{Ahlfors}):
\[
\sigma(z):=z\prod_{0\ne\omega\in\Lambda_{\omega_1,\omega_2}}
\bigg(1-\frac{z}{\omega}\bigg)\exp\bigg\{
\frac{z}{\omega}+\frac{z^2}{2\omega^2}
\bigg\}.
\]  
The function $\sigma(z)$ is entire, with periodicity-type formulae
\begin{equation}
\sigma(z+\omega_j)=-\sigma(z)\exp\big(\lambda_j(z+\tfrac{\omega_j}{2})\big),
\qquad j=1,2,
\label{eq-period1}
\end{equation}
where the constants are $\lambda_j:=2\zeta(\tfrac{\omega_j}2)$, as
expressed in terms of the logarithmic derivative 
$\zeta(z):=\sigma'(z)/\sigma(z)$ known as the Weierstrass zeta function. 
The relationship with the classical Weierstrass function is 
$\zeta'(z)=-\wp(z)$.
We consider in the planar zero packing problem the function $f$
\[
f(z):=a\,\e^{\xi z+\eta z^2}\sigma(z),
\]
where $a$ is a positive amplitude constant, and $\xi,\eta\in\C$ are parameters
to be determined. We would like the associated function
\begin{equation}
\e^{-|z|^2}|f(z)|=a\,\e^{-|z|^2+\re(\xi z+\eta z^2)}|\sigma(z)|
\label{eq-sigmafunction1}
\end{equation}
to be periodic with the two complex periods $\omega_1,\omega_2$. 
This is only possible if the density of the lattice $\Lambda_{\omega_1,\omega_2}$
has a normalized area of the fundamental rhombus $\mathcal{D}$ which equals 
$\frac12$. In terms of the periods, the requirement is that 
\[
\omega_1\im \omega_2=\frac{\pi}{2},
\]
which is related with the classical Legendre relation 
$\lambda_1\omega_2-\lambda_2\omega_1=\imag2\pi$.
Under this condition, it is indeed possible to specify values for the constants
$\xi,\eta$ such that the function \eqref{eq-sigmafunction1} gets to be
doubly periodic. If we choose 
\[
\omega_1:=\frac{\pi^{1/2}}{3^{1/4}},\quad\omega_2:=\frac{\pi^{1/2}}{3^{1/4}}
\e^{\imag \pi/3},
\]
which makes $\Lambda_{\omega_1,\omega_2}$ an equilateral triangular tiling of
the plane, 
this condition is fulfilled, and the appropriate values of the constants 
$\xi,\eta$ are then
%Korrektion tag bort faktorn 2 till höger
\[
\xi=0,\quad \eta=1-\frac{\zeta(\frac{\omega_1}{2})}{\omega_1}=
\frac{\bar\omega_2}{\omega_2}-\frac{\zeta(\frac{\omega_2}{2})}{\omega_2}.
\]
The asymptotic discrepancy density associated with this
particular choice can then be calculated over a single fundamental rhombus
$\mathcal{D}$ for the tiling $\C/\Lambda_{\omega_1,\omega_2}$,  
\begin{multline}
\lim_{R\to+\infty}\frac{1}{R^2}\int_{\D(0,R)}\big(\e^{-|z|^2}|f(z)|-1\big)^2
\diff A(z)=\frac{1}{|\mathcal{D}|_A}\int_{\mathcal{D}}
\big(\e^{-|z|^2}|f(z)|-1\big)^2\diff A(z)
\\
=\frac{1}{|\mathcal{D}|_A}\int_{\mathcal{D}}
\big(a\,\e^{-|z|^2+\re(\eta z^2)}|\sigma(z)|-1\big)^2\diff A(z),
\label{eq-asdensityW1}
\end{multline}
where we are free to minimize over the parameter $a$. Here, as mentioned
previously, $|\mathcal{D}|_A=\frac12$ is the normalized area of the 
fundamental rhombus.
The right-hand side of \eqref{eq-asdensityW1} is in a natural sense the
average of $\Psi_f$ over the torus $\C/\Lambda_{\omega_1,\omega_2}$. 

\begin{rem}
%{\rm(a)} 
Numerical implementation of the above integral 
\eqref{eq-asdensityW1}, minimized over the parameter $a$, 
was carried out by Wennman \cite{Wenn2} using 
{\sc Mathematica}, which resulted in the value $0.061203\ldots$, so that in
particular, $\rho_\C\le 0.061203\ldots$. We suggest that this
inequality is actually an equality.
\label{rem-monopole}
\end{rem}

\subsection{The stochastic minimization approach to 
planar zero packing}

It is difficult to know offhand what kind of packing of zeros would 
be optimal for the calculation of the asymptotic minimal discrepancy density 
$\rho_\C$. A reasonable approach is to let a stochastic process do
the digging for the optimal configuration, as in the so-called \emph{Bellman
function method}, exploited repeatedly in harmonic analysis (see, e.g., the
survey \cite{Ose}). First, we note that the assumption that the function
$f$ should be a polynomial in \eqref{eq-packR} is excessive, since polynomials
are dense in many spaces of holomorphic functions. In particular, the
minimal local density $\rho_\C(R)$ is unperturbed if we minimize e.g. over 
all entire functions $f$. Here, we will replace $f$ by a Gaussian analytic 
function (GAF) with close-to-optimal behavior. To set the notation, we let 
$N_\C(0,1)$ stand for the standard rotationally invariant Gaussian 
distribution with probability measure $\e^{-|\zeta|^2}\diff A(\zeta)$ in the 
plane $\C$. 
We pick independent copies $\xi_j\in N_\C(0,1)$ for $j=0,1,2,\ldots$, and
let $F$ be the GAF process \cite{HKPV}
\[
F(z):=\sum_{j=0}^{+\infty}\frac{\xi_j}{\sqrt{j!}}\, 2^{j/2}z^j,\qquad z\in\C.
\]
%where $c$ is a positive scaling constant that we may adjust. 
The way things are set up,  
\[
F(z)\e^{-|z|^2}
\]
has the same standard complex normal distribution $N_\C(0,1)$ irrespective of 
the point $z\in\C$. 
Given a positive amplitude constant $b$, we observe that the associated density
\[
\rho_{bF}(R):=\frac{1}{R^2}\int_{\D(0,R)}\Psi_{bF}(z)\diff A(z)
=\frac{1}{R^2}\int_{\D(0,R)}(b|F(z)|\e^{-|z|^2}-1)^2\diff A(z)
\]
is stochastic, and we may ask for the number  
\[
\rho_\C^{\mathrm{prob}}(R):=
\inf\big\{t>0:\,\exists\, b>0\,\,\text{such that}\,\,
\mathbb{P}(\rho_{bF}(R)\le t)>0\big\}
\]
where $\mathbb{P}(e)$ stands for the probability of the event $e$. Then 
clearly, $\rho_\C(R)\le\rho_\C^{\mathrm{prob}}(R)$, and we actually have equality.
%conjecture the following.
\medskip

\begin{prop}
%\emph{
We have that $\rho_\C(R)=\rho_\C^{\mathrm{prob}}(R)$, and hence that
$\rho_\C=\liminf_{R\to+\infty}\rho_\C^{\mathrm{prob}}(R)$.
%}
\label{propconj-1}
\end{prop} 

\begin{proof}[Proof sketch]
We will fix the parameter $b:=1$, which only makes things harder.
Since every holomorphic $f$ modulo $\mathrm{O}(z^{N+1})$  occurs with 
positive density in the process $F(z)$ (i.e., every finite sequence of the 
first $N$ Taylor coefficients occurs with positive
density in the stochastic sequence $2^{j/2}\xi_j/\sqrt{j!}$, $j=0,\ldots,N$), 
and for fixed $R$, the infimum in \eqref{eq-packR} is almost achieved by  
polynomials of sufficiently high degree, we can conclude that 
$\rho_\C(R)=\rho_\C^{\mathrm{prob}}(R)$ should hold. 
The influence of the remaining stochastic Taylor
coefficients $2^{j/2}\xi_j/\sqrt{j!}$ for $j>N$ to the stochastic integral 
$\rho_{F}(R)=\rho_{bF}(R)$ can be shown to be insignificant for big enough $N$. 
\end{proof}

Let $\expect$ stand for the expectation, and observe that
\begin{multline*}
\expect\rho_{bF}(R)=\frac{1}{R^2}\int_{\D(0,R)}\expect\Psi_{bF}(z)\diff A(z)
=\frac{1}{R^2}\int_{\D(0,R)}(b^2\e^{-2|z|^2}\expect|F(z)|^2
-2b\e^{-|z|^2}\expect|F(z)|+1)\diff A(z)
\\
=\frac{1}{R^2}\int_{\D(0,R)}(b^2-b\sqrt{\pi}+1)\diff A(z)
=b^2-b\sqrt{\pi}+1=(b-\tfrac{1}{2}\sqrt{\pi})^2+1-\frac{\pi}{4},
\end{multline*}
which tells us that the expected value of $\rho_{bF}(R)$ is minimized for 
the amplitude $b=\frac12\sqrt{\pi}$, and that the minimal expected value
equals $1-\frac{\pi}{4}=0.214\ldots$. We obtain immediately an upper bound for
$\rho_\C$:

\begin{prop}
We have the following bounds: 
\[
\rho_\C=\liminf_{R\to+\infty}\rho_\C^{\mathrm{prob}}(R)\le
\liminf_{R\to+\infty}\min_{b>0}\expect\rho_{bF}(R)=1-\frac{\pi}{4}.
\]
\end{prop}

It is of course na\"\i{}ve to believe that a simple expectation calculation 
would supply strong information. However, if we could get a grasp of 
the higher moments $\expect(\rho_{bF}(R))^k$ for $k=2,3,4,\ldots$ things 
would be different. This is related with the ``moment support bounding 
problem''.

\begin{rem}
By the planar analogues of the methods we develop in Section 
\ref{sec-appendix} for the hyperbolic setting, it can established that 
$\rho_\C>0$. 
\label{rem-4.3.3}
\end{rem}

%\end{document}
\subsection{Hyperbolic zero packing}

We now return to the hyperbolic zero packing problem. It is as before 
related to the possible improvement in the Cauchy-Schwarz inequality, in line
with Lemma \ref{lem-basic0}. This time, the discrepancy is given by
\[
\Phi_{f}(z):=\big((1-|z|^2)|f(z)|-1\big)^2,\qquad z\in\D,
\]
for a polynomial $f$, or more generally, $f$ which is holomorphic in $\D$. 
Again, $\Phi_f(z)=0$ is the same as the equality $(1-|z|^2)|f(z)|=1$ 
which has no holomorphic solution $f$.  
The reason is the same as before: $\log|f|$ is harmonic off the zeros 
of $f$, while $\hDelta\log\frac{1}{1-|z|^2}=(1-|z|^2)^{-2}>0$.
The average density of $\Phi_f$ with respect to the hyperbolic area element
$\diff A_{\mathbb{H}}(z):=(1-|z|^2)^{-2}\diff A(z)$ is the ratio
\[
\frac{\int_{\D(0,r)}\Phi_f\diff A_{\mathbb{H}}}{\int_{\D(0,r)}\diff A_{\mathbb{H}}}
=\frac{\int_{\D(0,r)}\Phi_f\diff A_{\mathbb{H}}}{\frac{r^2}{1-r^2}}
\]
and we could consider the inf over $f$ and then the liminf as $r\to1^-$.
However, since in hyperbolic geometry the length of boundary of $\D(0,r)$ is
substantial, the cutoff is a bit rough. To reduce the boundary effects, 
we instead average further before taking the ratio (compare, e.g. with
Seip's densities \cite{Seip}),
\begin{equation}
\frac{\int_0^r\int_{\D(0,t)}\Phi_f\diff A_{\mathbb{H}}\frac{\diff t}{t}}
{\int_0^r\int_{\D(0,t)}\diff A_{\mathbb{H}}\frac{\diff t}{t}}
=\frac{\int_{\D(0,r)}\Phi_f(z)(1-|z|^2)\diff A_{\mathbb{H}}(z)}
{\log\frac{1}{1-r^2}}.
\label{eq-hypaverage}
\end{equation}
So, the minimal average discrepancy we are after is, for $0<r<1$,
\begin{equation}
\rho_{\mathbb{H}}(r):=\frac{1}{\log\frac{1}{1-r^2}}
\inf_f\int_{\D(0,r)}\Phi_{f}(z)\frac{\diff A(z)}{1-|z|^2}=
\frac{1}{\log\frac{1}{1-r^2}}
\inf_f\int_{\D(0,r)}\big((1-|z|^2)|f(z)|-1\big)^2\frac{\diff A(z)}{1-|z|^2},
\label{eq-rho3}
\end{equation}
where the infimum is over all polynomials $f$, or, which gives the same 
result, over all holomorphic functions $f:\D\to\C$. In view of Lemma 
\ref{lem-basic0},
this discrepancy is the best constant for the improved Cauchy-Schwarz 
inequality
\begin{equation}
\bigg\{\int_{\D(0,r)}|f|\diff A\bigg\}^2\le (1-\rho_{\mathbb{H}}(r))
\log\frac{1}{1-r^2}\times\int_{\D(0,r)}|f(z)|^2(1-|z|^2)\diff A(z).
%\log\frac{1}{1-r^2}
\label{eq-CS-improvhyp1}
\end{equation}

\begin{defn}
For the above problem, the 
the \emph{minimal discrepancy density for hyperbolic zero packing} is 
$\rho_{\mathbb{H}}:=\liminf_{r\to1^-}\rho_{\mathbb{H}}(r)$.
\label{defn-odp2}
\end{defn}

\begin{prob}
Determine the value of $\rho_{\mathbb{H}}$. For which  
configurations of zeros of the function $f$ is it asymptotically attained? 
Is a lattice configuration optimal asymptotically?
\label{prob-2.001}
\end{prob}

Although the zero packing problem involves global issues, it probably has
some analogies with the more local hyperbolic circle packing problems (see,
e.g., \cite{Steph}).

\subsection{Hyperbolic Sch\"afli tilings}
\label{subsec-tiling}

One strategy for hyperbolic zero packing would be to pack according to a 
lattice configuration, for instance given by a tiling of the disk by 
hyperbolic regular $p$-gons with $q$ tiles meeting at each vertex 
(provided $p,q\ge3$). 
We illustrate with a fourfold octagonal ($p=8,q=4$) tiling of Figure 
\ref{fig-3}. Such a \emph{Sch\"afli tiling} exists provided that 
$a_{p,q}:=\frac14(p-2-\frac{2p}{q})>0$, and then the hyperbolic 
$\diff A_{\mathbb{H}}$-area of the $p$-gon is precisely $a_{p,q}$.  
A Sch\"afli tile is not always a fundamental domain for a Fuchsian group
$\Gamma$, as this happens if and only if the Poincar\'e cycle condition is 
fulfilled (see \cite{Maskit}). 

We are particularly interested in a Sch\"afli tiling which has normalized
area $a_{p,q}:=\frac14(p-2-\frac{2p}{q})=\frac12$, because this is 
analogous to what we saw with the lattice tiling of Subsection 
\ref{subsec-packingproblem1}, and would allow us to fit in exactly one zero per
tile, located at the hyperbolic center point of each tile. This area
condition can be written in the form 
\[
\frac{4}{p}+\frac{2}{q}=1,
\]
which has positive integer solutions $(p,q)$ of the form $(5,10)$, $(6,6)$,
$(8,4)$, and $(12,3)$, and generalized solutions $(4,\infty)$ and 
$(\infty,2)$. In particular, the $(8,4)$ tiling of  Figure \ref{fig-3} has
tiles with $\diff A_{\mathbb{H}}$-area $\frac12$. Such a tiling cannot 
correspond to a fundamental domain because the Poincar\'e cycle condition 
is not fulfilled. However, if we really want to, we can still glue together
the edges of the octagon in the standard fashion (which means that every
other edge gets glued pairwise, cyclically), but the resulting compact 
surface then obtains an irregular point with angle $4\pi$ around it 
(we might call it a \emph{branching point}, a \emph{ramified point}, or a 
\emph{flabby cone point}).  
 Another rather immediate way to see it is to use the Gauss-Bonnet 
theorem, which gives that the $\diff A_{\mathbb{H}}$-area of a fundamental
domain equals the integer $g-1\ge1$, where $g$ is the genus of the 
corresponding compact Riemann surface. 

\begin{figure}[h]
\includegraphics
[
trim = 0cm 18cm 0cm 18cm, clip=true,
%trim = 0cm 8cm 0cm 8cm, clip=true,
%totalheight=0.42\textheight
height=5.4cm, width=5.3cm]
{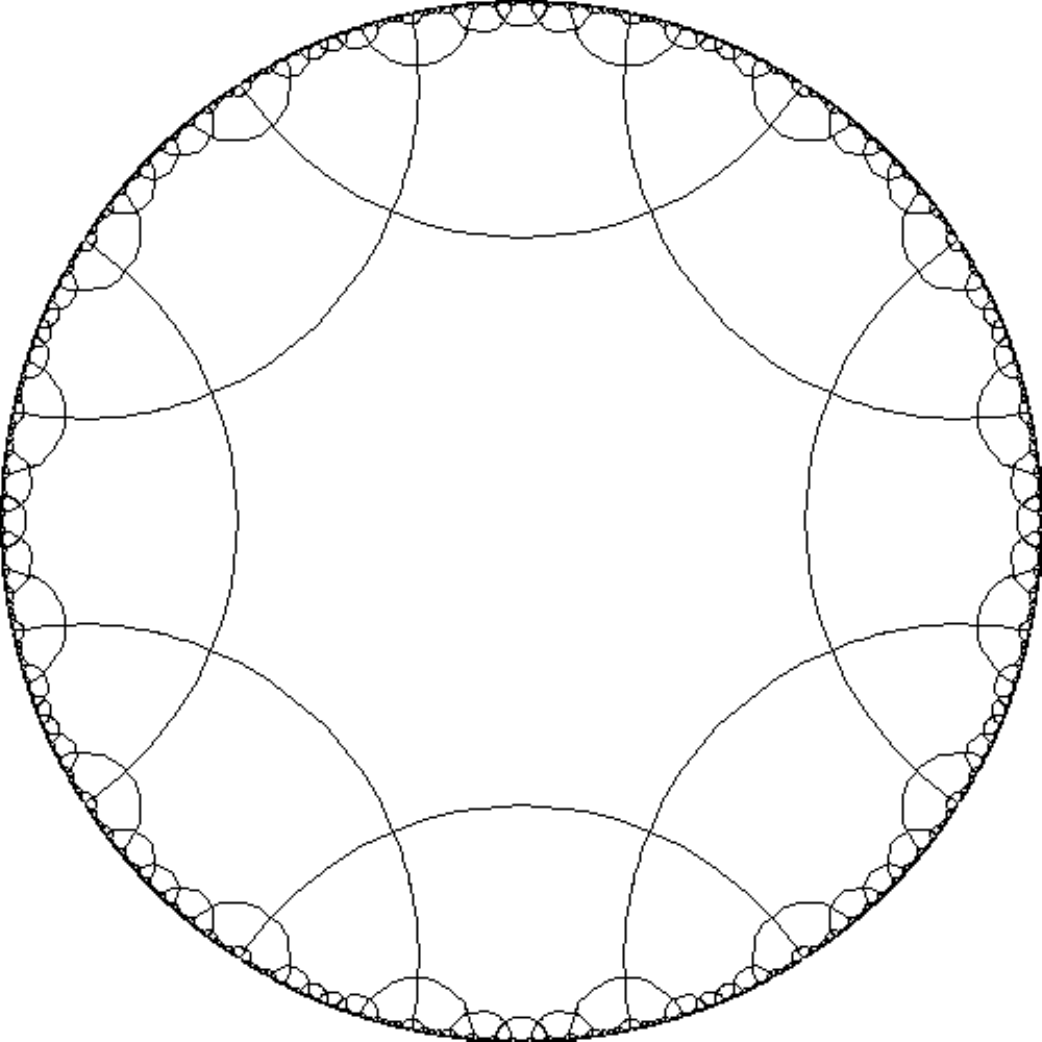}
\caption{
%Illustration of the fivefold pentagonal tiling.
Illustration of the fourfold octagonal tiling $(p,q)=(8,4)$.
}
\label{fig-3}
\end{figure}

\subsection{The stochastic minimization approach to 
hyperbolic zero packing}
\label{subsec-stochhyper}

As in the planar case, it is difficult to know offhand what kind of 
packing of zeros would be optimal for the calculation of the asymptotic 
minimal discrepancy density $\rho_{\mathbb{H}}$. Again, a reasonable 
approach is to let a stochastic process do the digging for the optimal 
configuration, and we look for an appropriate GAF process to supply 
random holomorphic functions in $\D$. 
As before, we pick independent copies $\eta_j\in N_\C(0,1)$ for 
$j=0,1,2,\ldots$, and let $G$ be the GAF process 
\[
G(z):=\sum_{j=0}^{+\infty}\eta_j\sqrt{j+1}\,z^j,\qquad z\in\C.
\]
It is well-known that $(1-|z|^2)G(z)$ has complex normal distribution 
$N_\C(0,1)$ irrespective of the point $z\in\C$. 
Given a positive amplitude constant $b$, we observe that the associated 
density
\[
\rho_{bG}(r):=\frac{1}{\log\frac{1}{1-r^2}}\int_{\D(0,r)}\Phi_{bG}(z)
\frac{\diff A(z)}{1-|z|^2}
=\frac{1}{\log\frac{1}{1-r^2}}\int_{\D(0,r)}(b|G(z)|(1-|z|^2)-1)^2
\frac{\diff A(z)}{1-|z|^2}
\]
is stochastic, and we may ask for the number  
\[
\rho_{\mathbb{H}}^{\mathrm{prob}}(r)
:=\inf\big\{t>0:\,\exists\, b>0\,\,\text{such that}\,\,
\mathbb{P}(\rho_{bG}(r)\le t)>0\big\}.
\]
Then clearly, $\rho_{\mathbb{H}}(r)\le\rho_{\mathbb{H}}^{\mathrm{prob}}(r)$, 
and in analogy with Proposition \ref{propconj-1}, we have equality.

\begin{prop}
We have that $\rho_{\mathbb{H}}(r)=\rho_{\mathbb{H}}^{\mathrm{prob}}(r)$, and hence
$\rho_{\mathbb{H}}=\liminf_{r\to1^-}\rho_{\mathbb{H}}^{\mathrm{prob}}(r)$.
\label{propconj-2}
\end{prop} 

The proof is essentially identical to that of Proposition \ref{propconj-1},
and left to the reader.
As for the value of the asymptotic density $\rho_{\mathbb{H}}$, we observe that
\begin{multline}
\expect\rho_{bG}(r)=\frac{1}{\log\frac{1}{1-r^2}}
\int_{\D(0,r)}\expect\Phi_{bG}(z)\frac{\diff A(z)}{1-|z|^2}
\\
=\frac{1}{\log\frac{1}{1-r^2}}\int_{\D(0,r)}(b^2(1-|z|^2)^2\expect|G(z)|^2
-2b(1-|z|^2)\expect|G(z)|+1)\frac{\diff A(z)}{1-|z|^2}
\\
=\frac{1}{\log\frac{1}{1-r^2}}\int_{\D(0,r)}(b^2-b\sqrt{\pi}+1)
\frac{\diff A(z)}{1-|z|^2}
=b^2-b\sqrt{\pi}+1=(b-\tfrac{1}{2}\sqrt{\pi})^2+1-\frac{\pi}{4},
\end{multline}
which tells us that the expected value of $\rho_{bG}(R)$ is minimized for 
the amplitude $b=\frac12\sqrt{\pi}$, and that the minimal expected value
equals $1-\frac{\pi}{4}=0.214\ldots$. We obtain immediately an upper bound for
$\rho_{\mathbb{H}}$, which is the same as in the planar case. This bound is
substantially weaker than the one found by Astala, Ivrii, Per\"al\"a, and 
Prause in \cite{AIPP} ($\rho_{\mathbb{H}}\le0.12087$).  

\begin{prop}
We have the following bounds: 
\[
\rho_{\mathbb{H}}=\liminf_{r\to1^-}\rho_{\mathbb{H}}^{\mathrm{prob}}(r)\le
\liminf_{r\to1^-}\min_{b>0}\expect\rho_{bG}(r)=1-\frac{\pi}{4}.
\]
\label{prop-estabove1.01}
\end{prop}

\section{Geometric zero packing for exponent $\beta$}
\label{sec-beta}

In this section, we introduce, for a positive real $\beta$, the 
$\beta$-exponent analogues of the planar and hyperbolic
zero packing problems considered in Section \ref{sec-imprCS}. 

\subsection{Planar zero packing for exponent $\beta$}

We introduce the $\beta$-exponent deformation of the density $\rho_{\C}$.

\begin{defn}
For a positive real $\beta$, let $\rho_\beta(\C)$ be the density
\begin{equation}
\rho_\beta(\C):=\liminf_{R\to+\infty}\inf_{f}\frac{1}{R^2}\int_{\D(0,R)}
(|f(z)|^\beta\e^{-|z|^2}-1)^2\diff A(z),
\label{eq-rhobeta1.01}
\end{equation}
where the infimum is taken over all polynomials $f$. We call this number
$\rho_\beta(\C)$ the \emph{$\beta$-exponent minimal discrepancy density 
for planar zero packing}. 
\end{defn}

In view of Lemma \ref{lem-basic0}, the density $\rho_\beta(\C)$ may also
be expressed as the minimal average ratio
\begin{equation}
\frac{1}{1-\rho_\beta(\C)}=\liminf_{R\to+\infty}\inf_f
\frac{R^{-2}\int_{\D(0,R)}|f(z)|^{2\beta}\e^{-2|z|^2}\diff A(z)}
{\bigg\{R^{-2}\int_{\D(0,R)}|f(z)|^{\beta}\e^{-|z|^2}\diff A(z)\bigg\}^2}
\label{eq-minaveratio1}
\end{equation}
where again the infimum is taken over all polynomials $f$. The instance
$\beta=2$ is the model problem considered by Aftalion, Blanc, and Nier 
\cite{ABN}. This constitutes a mathematical simplification of the energy 
functional in  the groundbreaking physical work of Abrikosov on Bose-Einstein 
condensates and type II superconductors (see \cite{Abr}). For $\beta=2$, 
it is shown rigorously in \cite{ABN} that the equilateral triangular lattice 
(see Figure \ref{fig-1}) is optimal {among the lattices}, and moreover, it is 
also shown that the corresponding Bargmann-Fock space function $f$
solves the Bargmann-Fock analogue of the standing wave equation for the cubic
Szeg\H{o} equation (for the cubic Szeg\H{o} equation, see e.g. \cite{GG} and
\cite{Poc}). This Bargmann-Fock analogue is known as the \emph{lowest Landau
level equation} (or LLL-equation), see, e.g., \cite{GerThom}, but we might
also suggest the term \emph{cubic Bargmann-Fock equation}. We take a look 
at this matter in the following subsection (Subsection \ref{subsec-Szego}). 

\begin{conj}
(Abrikosov)
The equilateral triangular lattice is optimal for $\beta$-exponent planar zero
packing for each positive $\beta$.
\label{conj-betaAbrik}
\end{conj}

\begin{figure}[h]
\includegraphics
[
%trim = 0cm 15cm 0cm 15cm, 
clip=true,
%trim = 0cm 8cm 0cm 8cm, clip=true,
%totalheight=0.42\textheight
height=6cm, width=6cm]
{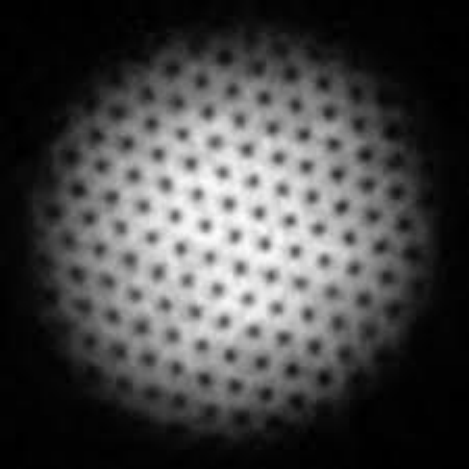}
%{hexagonal.pdf}
\caption{The equilateral triangular lattice (or honeycomb lattice) appears 
naturally in the physical context ($\beta=2$). The associated zeros are 
located inside the grayish dots.}
\label{fig-4.01}
\end{figure}

By a scaling transformation of the plane, it is easy to see that the density 
$\rho_\beta(\C)$ can be written as
\[
\rho_\beta(\C)=\liminf_{R\to+\infty}\inf_{f}\frac{1}{R^2}\int_{\D(0,R)}
(|f(z)|^\beta\e^{-\alpha|z|^2}-1)^2\diff A(z),
\] 
where the infimum is taken over all polynomials $f$, irrespectively of the 
value of the positive constant $\alpha$. 
Using this observation, we see that Conjecture \ref{conj-betaAbrik} maintains 
that 
\[
\rho_\beta(\C)=\frac{1}{|\mathcal{D}|_A}\int_{\mathcal{D}}
(|f(z)|^\beta\e^{-\beta|z|^2}-1)^2\diff A(z),\qquad 0<\beta<+\infty,
\] 
where $\mathcal{D}$ is the lattice rhombus and the function $f$ is defined 
in terms of the Weierstrass sigma function, as in \eqref{eq-asdensityW1}. 
We illustrate with the corresponding graph in Figure \ref{fig-4} communicated
by Wennman \cite{Wenn2}. For 
instance, the conjectured value for $\beta=2$ is $\rho_2(\C)=0.13763\ldots$, 
which corresponds to the 
number $\frac{1}{1-0.13763\ldots}=1.1596\ldots$ mentioned in Theorem 1.4 
of \cite{ABN}. Strictly speaking, Abrikosov did not quite go so far as
to Conjecture \ref{conj-betaAbrik}, but he did suggest it should be enough to
consider lattices and thought that the equilateral lattice was a natural 
candidate.

\medskip

\begin{figure}[h]
\includegraphics
[
%trim = 0cm 15cm 0cm 15cm, 
clip=true,
%trim = 0cm 8cm 0cm 8cm, clip=true,
%totalheight=0.42\textheight
height=4cm, width=7.5cm]
{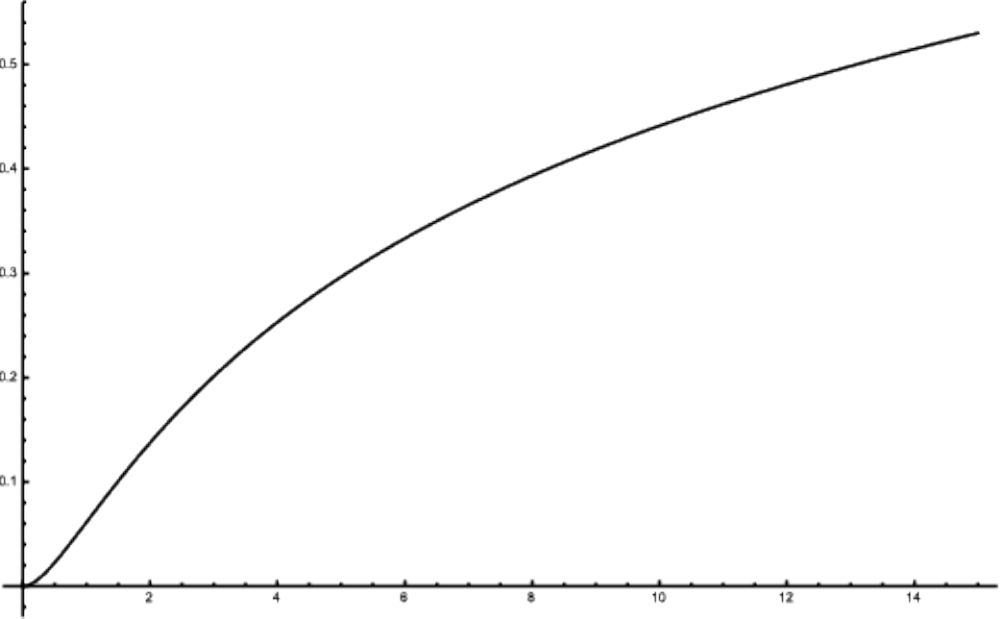}
%{hexagonal.pdf}
\caption{The graph of the density $\rho_\beta(\C)$ as a function of $\beta$
under Conjecture \ref{conj-betaAbrik}.}
\label{fig-4}
\end{figure}

As for possible monotonicity in the 
parameter $\beta$, we notice that trivially 
$\rho_{\beta}(\C)\le \rho_{k\beta}(\C)$ for $k=1,2,3,\ldots$. We can actually 
do better.

\begin{prop}
For positive reals $\beta,\beta'$ with $\beta<\beta'$, we have that 
$\rho_{\beta}(\C)\le \rho_{\beta'}(\C)$.
\label{prop-monotonic1}
\end{prop}

\begin{proof}
It is elementary that for $0<t<+\infty$, 
\[
(1-t^{\beta})^2\le(1-t^{\beta'})^2,\qquad 0<\beta<\beta'<+\infty,
\]
and, consequently, it follows that for any polynomial $f$, 
\[
\frac{1}{R^2}\int_{\D(0,R)}
(|f(z)|^\beta\e^{-\beta|z|^2}-1)^2\diff A(z)\le\frac{1}{R^2}\int_{\D(0,R)}
(|f(z)|^{\beta'}\e^{-\beta'|z|^2}-1)^2\diff A(z), \qquad  0<\beta<\beta'<+\infty.
\]
Together with the above-mentioned scaling invariance of the density
$\rho_\beta(\C)$, this gives the asserted monotonicity in $\beta$.
\end{proof}

\begin{rem}
Given the monotonicity, it is a natural question to ask what is the limit of 
$\rho_\beta(\C)$ as $\beta\to0^+$ and as $\beta\to+\infty$. We believe that 
\[
\lim_{\beta\to0^+}\rho_\beta(\C)=0,\quad\lim_{\beta\to+\infty}\rho_\beta(\C)=1.
\]
The first assertion is intuitively clear, since a sum of small point masses 
can approximate well a uniform distribution. This can probably form the 
backbone of a rigorous proof. The intuition behind the second assertion 
is that it should be impossible to reasonably approximate a uniform 
distribution using sums of very large point masses.     
\end{rem}

\subsection{Planar zero packing for exponent $\beta=2$ 
and the cubic Bargmann-Fock equation}
\label{subsec-Szego}

Suppose $f_0$ is a minimizer of the right-hand side integral in 
\eqref{eq-rhobeta1.01} for fixed $R$. We then use a variational argument 
comparing $f_0$ with $f_0+\epsilon h$ for a polynomial $h$ and an 
$\epsilon\in\C$ with $|\epsilon|$ tending to $0$ to show that 
\[
\Piop_{1}\bigg[(1-E_1|f_0|^\beta)\frac{|f_0|^\beta}{\bar f_0}1_{\D(0,R)}
\bigg]=0.
\]
This should be interpreted with some care for $0<\beta<1$ since it might then
be the case that $|f_0|^\beta/\bar f_0$ develops bad singularities at multiple 
zeros of $f_0$ (alternatively, a separate argument would be needed to rule out
multiple zeros). 
Here, $\Piop_{\alpha}$ is the Bargmann-Fock projection on the plane $\C$ with 
the Gaussian weight $E_\alpha(z):=\e^{-\alpha|z|^2}$. More explicitly, 
$\Piop_\alpha$ is given by
\[
\Piop_{\alpha}h(z):=\alpha\int_\C \e^{\alpha z\bar w}h(w)\,\e^{-\alpha|w|^2}\diff A(w),
\qquad z\in\C. 
\]
Expecting some kind stability as $R\to+\infty$, we naturally look for 
entire solutions $f_0$ with
\begin{equation}
\Piop_{1}\bigg[(1-E_1|f_0|^\beta)\frac{|f_0|^\beta}{\bar f_0}\bigg]=0.
\label{eq-equilib1}
\end{equation}
For $\beta=2$, the equation \eqref{eq-equilib1} just says that 
\begin{equation}
f_0=\Piop_{1}[E_1f_0|f_0|^{2}].
\label{eq-equilib1.0889}
\end{equation}
The \emph{cubic Bargmann-Fock equation} 
(or \emph{LLL-equation}) we alluded to in the preceding subsection is
\begin{equation}
\imag \partial_t u=\Piop_{1}[E_1u|u|^{2}],
\label{eq-cSzeqBF1}
\end{equation}
where $u=u(t,z)$ is assumed differentiable in $t$ and entire in $z$, and
such that the integral expression defining the right-hand side of 
\eqref{eq-cSzeqBF1} is well-defined. A \emph{stationary 
wave} (= a traveling wave with zero speed) is a solution of the form 
$u(t,z)=\e^{-\imag \omega t}f(z)$, where $\omega$ is a real constant and $f$
is entire. The equation \eqref{eq-cSzeqBF1} then reduces to
\begin{equation}
\omega f=\Piop_{1}[E_1f|f|^{2}],
\label{eq-cSzeqBF2}
\end{equation}
which for the value $\omega=1$ we recognize as the equation 
\eqref{eq-equilib1.0889}. Note that for the right-hand side of 
\eqref{eq-cSzeqBF2} to be well-defined, it is enough to assume that e.g. 
$|f(z)|=\Ordo(\e^{\eta|z|^2})$ as $|z|\to+\infty$ holds for some positive real 
$\eta<\frac23$. 
We should also point out the possibility to include some higher Landau levels
as well, as in \cite{HaiHed}. Indeed, the corresponding 
\emph{higher Landau level equation} analogous to \eqref{eq-cSzeqBF1} is
\begin{equation}
\imag \partial_t u=\Piop_{1}^{\langle N\rangle}[E_1u|u|^{2}],
\label{eq-cSzeqBF1.1}
\end{equation}  
where $\Piop_{1}^{\langle N\rangle}$ is the $N$-analytic Bargmann-Fock projection,
for $N=1,2,3,\ldots$. This is the orthogonal projection on the Gaussian
weighted space $L^2(\C,E_1)$ onto the subspace of $N$-analytic functions $v$, 
which solve the partial differential equation $\bar\partial^Nv=0$.

\subsection{Hyperbolic zero packing for exponent $\beta$ 
and field strength $\alpha$}
\label{hypzero-alphabeta}
We turn to the $\beta$-exponent analogue of the hyperbolic zero packing problem,
where we also introduce the positive real parameter $\alpha$, which in a sense 
corresponds to \emph{field strength}.

\begin{defn}
For a positive reals $\alpha,\beta$, let $\rho_{\alpha,\beta}(\mathbb{H})$ 
be the density
\[
\rho_{\alpha,\beta}(\mathbb{H}):=\liminf_{r\to1^-}\inf_{f}
\frac{1}{\log\frac{1}{1-r^2}}\int_{\D(0,r)}
\big((1-|z|^2)^\alpha|f(z)|^\beta-1\big)^2\frac{\diff A(z)}{1-|z|^2},
\] 
where the infimum is taken over all polynomials $f$. We call this number
$\rho_\beta(\C)$ the \emph{$\beta$-exponent minimal discrepancy density 
for hyperbolic zero packing with field strength $\alpha$}. 
\end{defn}

The choice of parameters $\alpha=\beta=1$ corresponds to the by now familiar 
density $\rho_{\mathbb{H}}$, that is, $\rho_{1,1}(\mathbb{H})=\rho_{\mathbb{H}}$. 
The analogue of  Proposition \ref{prop-monotonic1} in this hyperbolic 
context reads as follows.

\begin{prop}
For positive reals $\alpha,\alpha',\beta,\beta'$ with $\beta<\beta'$ and 
$k\frac{\alpha'}{\alpha}=\frac{\beta'}{\beta}$ for some $k=1,2,3,\ldots$, 
we have that 
$\rho_{\alpha,\beta}(\mathbb{H})\le \rho_{\alpha',\beta'}(\mathbb{H})$.
\label{prop-monotonic2}
\end{prop} 

\begin{proof}
For $k=1$, the proof essentially amounts to a repetition of the argument used 
in Proposition \ref{prop-monotonic1}. As for $k=2,3,4,\ldots$, we just need
to observe that for a holomorphic function $f$, its power $f^k$ is holomorphic 
as well, which gives the conclusion that $\rho_{\alpha,\beta}(\mathbb{H})\le
\rho_{\alpha,k\beta}(\mathbb{H})\le\rho_{\alpha',\beta'}(\mathbb{H})$, where the last
inequality follows from the $k=1$ case. 
\end{proof}

\begin{rem}
It follows from Proposition \ref{prop-monotonic2} that
the function $\beta\mapsto \rho_{\alpha\beta,\beta}(\mathbb{H})$ is 
monotonically increasing,
for fixed positive $\alpha$. It is then natural to ask for the limits 
as $\beta\to0^+$
and as $\beta\to+\infty$. We believe that
\[
\lim_{\beta\to0^+}\rho_{\alpha\beta,\beta}(\mathbb{H})=0,\quad
\lim_{\beta\to+\infty}\rho_{\alpha\beta,\beta}(\mathbb{H})=1.
\]
\end{rem}

It is however less clear what happens to $\rho_{\alpha,\beta}(\mathbb{H})$ 
if we let $\alpha\to+\infty$ and keep $\beta$ fixed. 

\begin{conj}
We believe that
\[
\lim_{\alpha\to+\infty}\rho_{\alpha,\beta}(\mathbb{H})=\rho_\beta(\C).
\]
More intuitively, only local effects become important as we increase the 
field strength $\alpha$.
\label{conj-converg1}
\end{conj}

Note that if we dilate the disk appropriately, the weight 
$(1-|z|^2)^\alpha$ becomes 
\[
\bigg(1-\frac{|z|^2}{\alpha}\bigg)^\alpha,
\]
which has the limit $\e^{-|z|^2}$ as $\alpha\to+\infty$. This shows the 
connection with the planar density. 

\begin{rem}
There is a variant of \eqref{eq-vareq1} which applies for more general 
$\alpha,\beta$. A minimizer $f_0$ for fixed $r$ meets 
\[
\Pop_{\alpha-1,r}\bigg[\big((1-|z|^2)^{\alpha}|f_0|^\beta-1\big)
\frac{|f_0|^{\beta}}{\bar f_0}\bigg]=0,
\]  
where $\Pop_{\alpha-1,r}$ is the weighted Bergman projection corresponding to the 
disk $\D(0,r)$ and the weight $(1-|z|^2)^{\alpha-1}$. As we noticed previously, 
the case when $0<\beta<1$ must be treated with additional care, as 
$|f_0|^\beta/\bar f_0$ may have nonintegrable singularities at zeros of $f_0$
of high multiplicity. Naturally, the instance $\alpha=\beta=1$ gives us 
back \eqref{eq-vareq1}. If $\beta=2$, the above equation says that
\[
f_0=\Pop_{\alpha-1,r}\big[(1-|z|^2)^{\alpha}f_0|f_0|^2\big],
\] 
and we are enticed to let $r\to1^-$, and consider the equation
\[
f_0=\Pop_{\alpha-1}\big[(1-|z|^2)^{\alpha}f_0|f_0|^2\big],
\]
where $\Pop_{\alpha-1}$ is the weighted Bergman projection on the unit 
disk $\D$ with the weight $(1-|z|^2)^{\alpha-1}$:
\[
\Pop_{\alpha-1} h(z):=\alpha\int_{\D}
\frac{(1-|w|^2)^{\alpha-1}}{(1-z\bar w)^{\alpha+1}}h(w)\diff A(w),\qquad z\in\D.
\] 
As in the preceding 
subsection, there is a corresponding time evolution equation
\[
\imag\partial_t u=\Pop_{\alpha-1}\big[(1-|z|^2)^{\alpha}u|u|^2\big],
\] 
which we understand as a hyperbolic geometry analogue of the LLL-equation 
\eqref{eq-cSzeqBF1}. 
\end{rem}

\section{Geometric zero packing for compact Riemann 
surfaces using logarithmic monopoles}
\label{sec-zeropack:logmono}

Our experience with geometric zero packing from Section \ref{sec-imprCS}
suggests a strong relation with regular configurations of lattice type, 
which suggests that the problem should be introduced on the quotient surface
level, which should then be a compact Riemann surface. Moreover, the notion 
of a \emph{logarithmic monopole} becomes very natural. It is the natural 
analogue of the Green function for the Laplacian in the context of compact 
surfaces.

\subsection{Logarithmic monopoles for compact Riemann 
surfaces}
\label{subsec-monopole3}

We consider a compact Riemann surface $\mathcal{S}$ with genus $g$, where 
$g\ge0$ is an integer. Then, by the uniformization theorem, $\mathcal{S}$ is
has one of the following forms: (i) if $g=0$, then $\mathcal{S}$ is 
topologically a sphere, which can be modelled by 
$\mathcal{S}=\mathbb{S}/\Gamma$ for a finite subgroup $\Gamma$ of the
automorphism group of the Riemann sphere $\mathbb{S}$, (ii) if $g=1$, then
$\mathcal{S}$ is a torus modelled by $\mathcal{S}=\C/\Lambda$ for a non-trival
lattice $\Lambda$, and (iii) if $g\ge2$, then $\mathcal{S}$ is modelled by
$\mathcal{S}=\mathbb{H}/\Gamma$, where $\mathbb{H}$ is the hyperbolic plane
and $\Gamma$ is a discrete subgroup of the automorphism group of $\mathbb{H}$.
In each of the cases (i)--(iii), we have a complete Riemannian metric with
constant curvature on the respective covering surfaces 
$\mathbb{S},\C,\mathbb{H}$, which then induces a canonical Riemannian metric 
on the surface $\mathcal{S}$. In a similar fashion, the canonical normalized
area measures $\diff A_{\mathbb{S}},\diff A,\diff A_{\mathbb{H}}$ induce a 
normalized area measure on $\mathcal{S}$, which we denote by 
$\diff A_{\mathcal{S}}$. The $\diff A_{\mathcal{S}}$-area of the whole surface
$\mathcal{S}$ is denoted by $a(\mathcal{S})$. 
The \emph{logarithmic monopole} $U(z,w)=U_{\mathcal{S}}(z,w)$, for points
$z,w\in\mathcal{S}$, is a real-valued function which for fixed $w$ has
\[
\hDelta^{\mathcal{S}}U(\cdot,w)=\frac12\delta_w-
\frac{1}{2a({\mathcal{S}})},
\]
where $\hDelta^{\mathcal{S}}$ is the normalized Laplace-Beltrami 
operator. The expression $\delta_w$ stands for the unit point mass at $w$,
treated as a $2$-form. The existence of this function is guaranteed by 
Corollary 8-2 of \cite{Spr}, which guarantees the existence of the 
corresponding logarithmic bipole $L(z,w,w')$ (see, e.g., \cite{Spr}, p. 213), 
which has a source at $w$ and a sink at $w'$. To obtain the monopole $U(z,w)$,
we just average this bipole function $L(z,w,w')$ with respect surface 
area in the $w'$ variable.
It is unique up to an additive real constant. 

For a nontrivial lattice $\Lambda$ with two generators, the torus 
$\C/\Lambda$ is of course a compact Riemann surface with genus $1$. In this
case, the logarithmic monopole may be expressed explicitly in terms of the 
classical Weierstrass sigma function (see Subsection 
\ref{subsec-monopole1} below). We will also consider the spherical 
genus $0$ case, as well as the (hyperbolic) genus $\ge2$ case.

\subsection{Geometric zero packing on compact 
Riemann surfaces}
We turn to the geometric zero packing problem for general compact surfaces.
We introduce the notation
\[
\langle f\rangle_{\mathcal{S}}:=\frac{1}{a(\mathcal{S})}\int_{\mathcal{S}}
f\diff A_{\mathcal{S}}
\]
for the surface average of a summable function $f:\mathcal{S}\to\C$. 
Let $U(z,w)=U_{\mathcal{S}}(z,w)$ denote the logarithmic monopole on the 
compact surface $\mathcal{S}$, normalized so that 
$\langle U(\cdot,w)\rangle_{\mathcal{S}}=0$.

\begin{defn}
For $n$ points $w_1,\ldots,w_n\in\mathcal{S}$, we write 
\[
U^{\langle n\rangle}(z):=U(z,w_1)+\cdots + U(z,w_n).
\]
For a positive real $\beta$, the \emph{minimal average discrepancy for 
geometric $\beta$-zero packing} on $\mathcal{S}$ is the sequence of numbers
\begin{equation}
\rho_{n,\beta}(\mathcal{S}):=\inf_{b,w_1,\ldots,w_n}
\big\langle
\big(b\,\e^{\beta U^{\langle n\rangle}}-1\big)^2\big\rangle_{\mathcal{S}},
\label{eq-rhosurf}
\end{equation}
where the infimum is over all positive reals $b$ and all points
$w_1,\ldots,w_n\in\mathcal{S}$. A collection of points 
$\{w_1,\ldots,w_n\}$ which realizes the infimum for some value of $b$ is 
called an \emph{equilibrium configuration} (for exponent $\beta$).
\end{defn}

We observe that by standard Hilbert space methods,
\[
\rho_{n,\beta}(\mathcal{S})^{1/2}=\inf_{b,w_1,\ldots,w_n}\sup_{g}
\big\langle
(b\,\e^{\beta U^{\langle n\rangle}}-1) g\,\big\rangle_{\mathcal{S}}
\]
where the supremum runs over all real-valued functions $g$ in the unit 
ball of $L^2(\mathcal{S})$. This is somewhat analogous to the  
Kantorovich-Wasserstein distance used in optimal transport \cite{Vill}.
The quantity $\rho_{n,\beta}(\mathcal{S})$ measures how evenly we can place the
$n$ points $w_1,\ldots,w_n$ on the surface so as to minimize the average 
discrepancy. 

As for monotonicity issues, the approach used in Propositions 
\ref{prop-monotonic1}
and \ref{prop-monotonic2} also shows the following. We suppress the 
analogous proof. 

\begin{prop}
Fix a compact Riemann surface $\mathcal{S}$ and an integer $n=1,2,3,\ldots$. 
Then the function $\beta\mapsto\rho_{n,\beta}(\mathcal{S})$ is monotonically 
increasing:
\[
\rho_{n,\beta}(\mathcal{S})\le\rho_{n,\beta'}(\mathcal{S}),\qquad 
0<\beta<\beta'<+\infty.
\]
\end{prop}

Regarding the possible convergence as $n\to+\infty$ for a fixed exponent 
$\beta$, we suggest the following conjecture.

\begin{conj}
We believe that 
\[
\lim_{n\to+\infty}
\rho_{n,\beta}(\mathcal{S})\to\rho_{\beta}(\C)
\]
for any fixed compact surface $\mathcal{S}$ and any fixed positive real 
$\beta$. 
\label{conj-converg2}
\end{conj}

To arrive at an equilibrium configuration $\{w_1,\ldots,w_n\}$, we might try
a numerical approach based on the gradient flow method.
For a positive real $\gamma>0$, we may think of 
\[
Z_\gamma(w_1,\ldots,w_n):=\langle \e^{\gamma U^{\langle n\rangle}}\rangle_{\mathcal{S}}
\]
as a marginal partition function for the $n$-point $\beta$-ensemble on 
the surface $\mathcal{S}$ (see \cite{KMMW}, also \cite{krish} for the 
sphere with $\gamma=2$), and associate to it the probability 
measure on the surface $\mathcal{S}$
\begin{equation}
\frac{1}{Z_\gamma(w_1,\ldots,w_n)}\e^{\gamma U^{\langle n\rangle}}
\frac{\diff A_{\mathcal{S}}}{a(\mathcal{S})}.
\label{eq-probmeas101}
\end{equation}
%Expectation with respect to this surface probability measure will be written
%$\expect^\gamma_{z_1,\ldots,z_n}$ (here, the points $z_1,\ldots,z_n$ are 
%thought of as fixed). 
We will return to this probability measure a little later, and meanwhile
observe that the optimal value of $b$ in the definition of 
$\rho_{n,\beta}(\mathcal{S})$ is
\begin{equation}
b=b(\beta)=\frac{\langle\e^{\beta U^{\langle n\rangle}}\rangle_{\mathcal{S}}}
{\langle\e^{2\beta U^{\langle n\rangle}}\rangle_{\mathcal{S}}}.
\label{eq-bnumber}
\end{equation}
Moreover, since 
\[
\inf_{b}
\big\langle
\big(b\,\e^{\beta U^{\langle n\rangle}}-1\big)^2
\big\rangle_{\mathcal{S}}
=\Bigg\langle\Bigg(
\frac{\langle\e^{\beta U^{\langle n\rangle}}\rangle_{\mathcal{S}}}
{\langle\e^{2\beta U^{\langle n\rangle}}\rangle_{\mathcal{S}}}
\,
\e^{\beta U^{\langle n\rangle}}-1\Bigg)^2\Bigg\rangle_{\mathcal{S}}
=1-\frac{\langle\e^{\beta U^{\langle n\rangle}}\rangle_{\mathcal{S}}^2}
{\langle\e^{2\beta U^{\langle n\rangle}}\rangle_{\mathcal{S}}},
\]
it is immediate that
\[
\rho_{n,\beta}(\mathcal{S})=1-\sup_{w_1,\ldots,w_n}
\frac{\langle\e^{\beta U^{\langle n\rangle}}\rangle_{\mathcal{S}}^2}
{\langle\e^{2\beta U^{\langle n\rangle}}\rangle_{\mathcal{S}}}.
\]
The gradient flow method suggests that to climb the hill to the top we should 
always move in the direction of steepest ascent. In this situation, it is 
possible to calculate that direction at a given $n$-tuple $(w_1,\ldots,w_n)$.
In particular, at the top where we stop the gradient vanishes. 
When we write this out, we obtain the following criterion.
%Now, application of the gradient flow method tells us that the $n$-tuple 
%of points $(z_1,\ldots,z_n)$ should move in the direction of biggest 
%increase, given by the vector
%\[
%\big(\expect^\beta_{z_1,\ldots,z_n} V(\cdot,z_1)-
%\expect^{2\beta}_{z_1,\ldots,z_n}
%V(\cdot,z_1),\ldots,
%\expect^\beta_{z_1,\ldots,z_n} V(\cdot,z_n)-\expect^{2\beta}_{z_1,\ldots,z_n}
%V(\cdot,z_n)\big).
%\] 
%Here, we write $V(z,w):=\bar\partial_w U(z,w)$, where we made a smooth choice
%of the free constant involved with $U(z,w)$, and we differentiate with
%respect to the global complex coordinate associated with the 
%universal covering surface (which is either the Riemann sphere, the complex
%plane, or the hyperbolic plane).   
%In particular, for an optimal configuration $z_1,\ldots,z_n$, the above
%vector coincides with the zero vector:

\begin{prop}
\label{prop-8.2.4}
If $\{w_1,\ldots,w_n\}$ is an equilibrium configuration for exponent $\beta$
on the compact Riemann surface $\mathcal{S}$, then
we have
\[
\frac{\big\langle \e^{\beta U^{\langle n\rangle}}\partial_w U
(\cdot,w)\big\rangle_{\mathcal{S}}}
{\langle \e^{\beta U^{\langle n\rangle}}\rangle_{\mathcal{S}}}
=\frac{\big\langle \e^{2\beta U^{\langle n\rangle}}\partial_w U
(\cdot,w)\big\rangle_{\mathcal{S}}}
{\langle \e^{2\beta U^{\langle n\rangle}}\rangle_{\mathcal{S}}},
\qquad w\in\{w_1,\ldots,w_n\}.
\]
\end{prop}

The expressions on the left-hand and right-hand sides of the displayed
equation of the above proposition \ref{prop-8.2.4} are the expectated values of
the function $\partial_w U(\cdot,w)$ with respect to the probability measure 
\eqref{eq-probmeas101} for $\gamma=\beta$ and $\gamma=2\beta$, respectively.
Note that the necessary condition for an equilibrium configuration stated 
in Proposition \ref{prop-8.2.4} differs from the corresponding condition 
for Fekete configurations \cite{ST} as well as from that of spherical designs, 
e.g. \cite{DGS}.
Note that the space of $n$-tuples $(w_1,\ldots,w_n)$ has complex dimension $n$,
while the proposition supplies $n$ complex nonlinear conditions 
(and hence $2n$ real conditions) which with some luck might have only a 
finite set of solutions.

\begin{rem}
(a)
In principle, the definition of geometric $\beta$-zero packing 
\eqref{eq-rhosurf} makes sense for negative exponents $\beta$ as well, 
provided that the exponent is not too large
negative: $-1<\beta<0$ is needed. It actually makes sense for $\beta=0$ as
well, if we minimize instead the quantity
\begin{equation}
\rho_{n,\star}(\mathcal{S}):=\inf_{w_1,\ldots,w_n}
\langle(U^{\langle n\rangle })^2\rangle_{\mathcal{S}},
\label{eq-rhosurf0.0}
\end{equation}
which formally arises as the limit of $\beta^{-2}\rho_{n,\beta}(\mathcal{S})$
as $\beta\to0$. In the context of \eqref{eq-rhosurf0.0} it is essential that 
we normalize the additive constant which we are allowed to add to $U(z,w)$ 
in such a way that 
\begin{equation}
\langle U(\cdot,w)\rangle_{\mathcal{S}}=0,\qquad w\in\mathcal{S}.
\label{eq-rhosurf0.1}
\end{equation}

\noindent (b) Let us compare the minimization problem \eqref{eq-rhosurf0.0}
with the Fekete problem which in the given setting asks for the 
configurations $w_1,\ldots,w_n$ that achieve the maximum 
\[
\sup_{w_1,\ldots,w_n}\sum_{j,k:j\ne k}U(w_j,w_k).
\]
It will be convenient to cut off the singularity in the logarithmic monopole
and replace it by a function $U_\epsilon(z,w)$ which is $C^2$-smooth in both 
variables and $\lim_{\epsilon\to0^+}U_\epsilon(z,w)=U(z,w)$.    
If the approximation is done well, with controlled errors, we could arrange it
so that for fixed $z\ne w$, we would have
\begin{multline*}
U(z,w)=2\int_{\mathcal{S}}U(\cdot,w)
\hDelta^{\mathcal{S}} U_\epsilon(\cdot,z)\diff A_{\mathcal{S}}+\mathrm{O}(\epsilon)
=2\int_{\mathcal{S}}U_\epsilon(\cdot,w)
\hDelta^{\mathcal{S}} U_\epsilon(\cdot,z)
\diff A_{\mathcal{S}}+\mathrm{O}(\epsilon)
\\
=-2\int_{\mathcal{S}}\partial^{\mathcal{S}} U_\epsilon(\cdot,z)
\bar\partial^{\mathcal{S}} U_\epsilon(\cdot,w)\diff A_{\mathcal{S}}
+\mathrm{O}(\epsilon).
\end{multline*}
Here we used the property \eqref{eq-rhosurf0.1}, and in a second step, 
Green's formula on the surface. In terms of notation, the differential 
operators $\partial^{\mathcal{S}}$ and $\bar\partial^{\mathcal{S}}$ are modified
to fit the surface (the conformal factor is inserted on the left-hand side).
It now follows that 
\begin{equation*}
\sum_{j,k:j\ne k}U(w_j,w_k)=-2\int_{\mathcal{S}}
\Big|\sum_j\partial^{\mathcal{S}} U_\epsilon(\cdot,w_j)\Big|^2
\diff A_{\mathcal{S}}+2\sum_j\int_{\mathcal{S}}
|\partial^{\mathcal{S}} U_\epsilon(\cdot,w_j)|^2\diff A_{\mathcal{S}}
+\mathrm{O}(\epsilon),
\end{equation*}
where the expression 
\[
\int_{\mathcal{S}}
|\partial^{\mathcal{S}} U_\epsilon(\cdot,w_j)|^2\diff A_{\mathcal{S}}
=\frac14\int_{\mathcal{S}}
|\nabla^{\mathcal{S}} U_\epsilon(\cdot,w_j)|^2\diff A_{\mathcal{S}}
\]
tends to $+\infty$ with rather precise asymptotics as $\epsilon\to0^+$ (at
least if $U_\epsilon(z,w)$ is chosen correctly). 
If we add a term to neutralize each such contribution 
(for each point $w_j$), the Fekete problem essentially asks for 
configurations that minimize the Dirichlet energy 
\[
\int_{\mathcal{S}}
|\nabla^{\mathcal{S}} U_\epsilon^{\langle n \rangle}|^2\diff A_{\mathcal{S}},
\quad\text{where}\quad 
U_\epsilon^{\langle n \rangle}:=U_\epsilon(\cdot,w_1)+\cdots+U_\epsilon(\cdot,w_n),
\] 
as $\epsilon\to0^+$ (see, e.g. \cite{Serf} for a more general situation). 
In comparison, the problem \eqref{eq-rhosurf0.0} asks us 
to minimize the corresponding Bergman energy (just the area-$L^2$ norm 
squared). 
\label{rem-energies}
\end{rem}

\subsection{Spherical zero packing} 
We briefly mention what happens when the compact Riemann surface has
\emph{genus} $0$. We will consider the Riemann sphere 
$\mathbb{S}:=\C_\infty=\C\cup\{\infty\}$ with the standard metric 
$\diff s_{\mathbb{S}}:=(1+|z|^2)^{-1}|\diff z|$, which has constant positive 
Gaussian curvature. The associated spherical normalized area measure is 
$\diff A_{\mathbb S}(z):=(1+|z|^2)^{-2}\diff A(z)$. Moreover, the associated 
logarithmic monopole is the function 
\[
U(z,w):=\log|z-w|-\frac12\log(1+|z|^2)+A(w),\qquad z,w\in\C,\,\,z\ne w,
\]
where we are free to choose the real number $A(w)$. The choice 
$A(w):=-\frac12\log(1+|w|^2)$ would seem to be the most appropriate, since
it gives the symmetry property $U(z,w)=U(w,z)$, typical of Green 
functions. Then we may define $U(z,w)$ for $w=\infty$ as well: 
$U(z,\infty)=-\frac12\log(1+|z|^2)$. After all, the basic property of the 
logarithmic monopole is that 
\[
\hDelta^{\mathbb{S}} U(\cdot,w)
=\frac12\delta_w-\frac12
\]
in the sense of distribution theory. 
Here, it is of course important that the area of the sphere is normalized to be
$a(\mathbb{S})=1$.

Let us look at the minimal average discrepancy for spherical $\beta$-zero 
packing, that is, the numbers $\rho_{n,\beta}(\mathbb{S})$, for $n=1$ and 
$n=2$. Since 
\begin{equation}
\rho_{n,\beta}(\mathbb{S})
=\inf_{a,w_1,\ldots,w_n}\int_{\mathbb{S}}
\bigg(a\,\frac{|z-w_1|^\beta\cdots|z-w_n|^\beta}
{(1+|z|^2)^{n\beta/2}(1+|w_1|^2)^{\beta/2}\cdots
(1+|w_n|^2)^{\beta/2}}-1\bigg)^2\diff A_{\mathbb{S}}(z),
\label{eq-rhosphere}
\end{equation}
where the infimum is over all positive reals $a$ and all complex numbers
$w_1,\ldots,w_n$, we calculate that (with $w_1=0$)
\[
\rho_{1,\beta}(\mathbb{S})=\inf_a\int_{\mathbb{S}}
\bigg(\frac{a|z|^\beta}{(1+|z|^2)^{\beta/2}}-1\bigg)^2
\diff A_{\mathbb{S}}(z)=\frac{\beta^2}{(2+\beta)^2}.
\]
As for $n=2$, it is intuitively clear that the two points should be 
antipodal in the optimal configuration, and we then pick $w_1=0,w_2=\infty$. 
As a consequence, 
\[
\rho_{2,\beta}(\mathbb{S})=\inf_a\int_{\mathbb{S}}
\bigg(\frac{a|z|^\beta}{(1+|z|^2)^{\beta}}
-1\bigg)^2\diff A_{\mathbb{S}}(z)=
1-\frac{2^{-4\beta}\pi^2\Gamma(2+2\beta)}
{(1+\beta)^2\Gamma(\frac{1+\beta}{2})^4},
\]
which we may compare with $\rho_{1,\beta}(\mathbb{S})$. Computer work 
strongly suggests that $\rho_{1,\beta}(\mathbb{S})\ge\rho_{2,\beta}(\mathbb{S})$ 
for all positive values of $\beta$ (this is from a calculation made by Wennman 
\cite{Wenn2}). For instance, with 
$\beta=1$, we find that $\rho_{1,1}(\mathbb{S})=0.111\ldots$, whereas
 $\rho_{2,1}(\mathbb{S})=0.07472\ldots$, which is much smaller and 
considerably closer to the conjectured value of $\rho_\C$ 
(which is $0.061203\ldots$, see Remark \ref{rem-monopole}). Here one 
might na\"\i{}vely guess that the function 
$n\mapsto \rho_{n,\beta}(\mathbb{S})$ is decreasing for fixed $\beta$. 
While this may be true for small $\beta$, it is certainly false for 
large $\beta$, as evidenced by further numerical work for $n=3$.
Compare also with Conjecture \ref{conj-converg2}. 

\subsection{Logarithmic monopoles for a torus}
\label{subsec-monopole1}

We turn to the case of a compact Riemann surface with \emph{genus $1$}.
Such a Riemann surface is a torus, and can be modelled by 
$\C/\Lambda_{\omega_1,\omega_2}$, in the notation of Subsection 
\ref{subsec-packingproblem1}.
Note that if we take logarithms in \eqref{eq-sigmafunction1}, 
we obtain the real-valued function 
\[
U(z):=\log(\e^{-|z|^2}|f(z)|)=-|z|^2+\log|f(z)|=-|z|^2+\log a+\re(\xi z+\eta z^2)
+\log|\sigma(z)|,
\]
and, we may define, more generally, $U(z,w):=U(z-w)$. Since the positive
constant $a$ is free, the function $U(z,w)$ is real-valued and well-defined 
up to an additive constant. Moreover, it is $\Lambda_{\omega_1,\omega_2}$-periodic
in both $z$ and $w$, with Laplacian 
\[
\hDelta U(\cdot,w)=-1+\frac12\sum_{\lambda\in\Lambda_{\omega_1,\omega_2}}
\delta_{\lambda+w}
\]  
in the sense of distribution theory, where $\delta_\xi$ is the unit point 
mass at the point $\xi\in\C$.

\subsection{Logarithmic monopoles for higher genus surfaces 
and character-modular forms}
\label{subsec-monopole2.1}
We turn to the case when the Riemann surface $\mathcal{S}$ has genus $g\ge2$.
We then equip the surface
with a metric of constant negative curvature, and use the hyperbolic plane
$\mathbb{H}$ as the universal covering surface. We model the hyperbolic plane
$\mathbb{H}$ by the unit disk $\D$ with the Poincar\'e metric.
This gives us the identification $\mathcal{S}\cong\D/\Gamma$, where 
$\Gamma$ is a Fuchsian group of M\"obius automorphisms. We write $D_\Gamma$
for a corresponding fundamental polygon bounded by hyperbolic geodesic
segments. We denote by $a(\Gamma)$ the $\diff A_{\mathbb{H}}$-area of $D_\Gamma$,
which is the same as the corresponding area of the surface $a(\mathcal{S})$.  
We first relate two properties of periodicity type.

\begin{prop}
Suppose $f:\D\to\C$ is holomorphic. Then, for real $\alpha$, the following 
are equivalent:
\smallskip

\noindent{\rm(a)} For all $\gamma\in\Gamma$, we have
$|\gamma'(z)|^\alpha |f\circ\gamma(z)|=|f(z)|$ on the disk $\D$.
\smallskip

\noindent{\rm(b)} The function $F(z):=(1-|z|^2)^\alpha|f(z)|$ is 
$\Gamma$-periodic, that is, $F\circ\gamma=F$ for all $\gamma\in\Gamma$.
\label{prop-periodic1.1}
\end{prop}

\begin{proof}
This is an immediate consequence of the identity 
\begin{equation}
1-|\gamma(z)|^2=(1-|z|^2)|\gamma'(z)|,
\label{eq-automid1.01}
\end{equation}
which holds for any M\"obius automorphism $\gamma\in\mathrm{aut}(\D)$.
\end{proof}

We want to analyze the property (a) of Proposition \ref{prop-periodic1.1}
more carefully. First, for an automorphism $\gamma\in\mathrm{aut}(\D)$,
the derivative $\gamma'$ is nonzero, which permits us to define its
logarithm $\log\gamma'$ holomorphically in $\D$ (any two choices will differ 
by an integer multiple of $\imag 2\pi$). The group $\Gamma$ is finitely
generated, and we pick generators $\gamma_1,\ldots,\gamma_m$, and choose
the corresponding logarithms $\log\gamma'_j$ for $j=1,\ldots,m$ as holomorphic
functions any way we like (the freedom is up to constants in $\imag 2\pi\Z$). 
If the group $\Gamma$ were free, we would then proceed to represent an 
arbitrary element $\gamma\in\Gamma$ as a ``word'' (a finite composition of the 
generators), and 
let the logarithm $\log\gamma'$ be determined by the natural property that
\begin{equation}
\log{(\tilde\gamma \circ\gamma)'}=(\log\tilde\gamma')\circ\gamma+\log\gamma',
\qquad \gamma,\tilde\gamma\in\Gamma.
\label{eq-logs1}
\end{equation}
In a second step, the powers 
\begin{equation}
(\gamma'(z))^{\alpha}=\exp(\alpha\log\gamma'(z)),\qquad \gamma\in\Gamma,
\label{eq-powers1}
\end{equation}
would be defined as well, and we would have the property that for $\alpha\in\C$,
\begin{equation}
((\tilde\gamma \circ\gamma)')^\alpha=((\tilde\gamma')^\alpha\circ\gamma)
(\gamma')^\alpha,\qquad \gamma,\tilde\gamma\in\Gamma.
\label{eq-powers2}
\end{equation}
However, typically $\Gamma$ is not a free group, as there are finitely many 
so-called relations which need to be satisfied as well (a relation is the 
condition that a nontrivial combination of the generators equals the 
identity). These relations will put some restraints on our freedom of choosing
the logarithms $\log\gamma'_j$ for our generators $\gamma_j$, and it may happen
that this cannot be done in a manner consistent with \eqref{eq-logs1}. 
One way out is to express each group element 
$\gamma\in\Gamma$ uniquely as a combination of generators by picking the
shortest ``word'' expressing the element (if there is compretition, just
pick one of the minimal ``words''), and to use \eqref{eq-logs1} along the
``word'' to define properly the logarithm of the derivative. Then we sacrifice
the property \eqref{eq-logs1} globally, and hence \eqref{eq-powers2} need not
hold for an arbitrary $\alpha\in\C$. However, equality in \eqref{eq-powers2}
holds automatically for integers $\alpha\in\Z$.   

A $\Gamma$-\emph{character} is a function $\chi:\Gamma\to\Te$ with the 
multiplicative property
$\chi(\tilde\gamma\circ\gamma)=\chi(\tilde\gamma)\chi(\gamma)$ for all
$\gamma,\tilde\gamma\in\Gamma$. 
We shall require a slightly generalized version of this notion, which we
refer to as a $(\Gamma,q)$-root character, or simply a $q$-root character if
the group $\Gamma$ is taken for granted. 
For a positive integer $q$, we say that
$\chi:\Gamma\to\Te$ is a \emph{$(\Gamma,q)$-root character} if and only if 
$\chi^q$ is a $\Gamma$-character (here, $\chi^q(\gamma)=(\chi(\gamma))^q$ 
is the $q$-th power). This generalizes the concept of the characters, since
a character is automatically a $q$-root character for integers 
$q=2,3,4,\ldots$.

\begin{prop}
Suppose $f:\D\to\C$ is holomorphic. Then, for rational $\alpha=p/q$, where
$p,q$ are coprime integers and $q>0$, the following 
are equivalent:
\smallskip

\noindent{\rm(a)} For all $\gamma\in\Gamma$, we have
$|\gamma'(z)|^\alpha |f\circ\gamma(z)|=|f(z)|$ on the disk $\D$.
\smallskip

\noindent{\rm(b)} There exists a $(\Gamma,q)$-root character $\chi$ such that 
$(\gamma'(z))^\alpha f\circ\gamma(z)=\chi(\gamma)f(z)$.
\label{prop-periodic1.2}
\end{prop}
 
\begin{proof}
Clearly, $(\mathrm{b})\implies(\mathrm{a})$, so we will obtain the remaining
implication $(\mathrm{a})\implies(\mathrm{b})$.
We observe from (a) that for given $\gamma\in\Gamma$, the two holomorphic 
functions $(\gamma')^\alpha f\circ\gamma$ and $f$ have the same modulus. This
is only possible if one is a unimodular constant times the other, that is,
\[
(\gamma')^\alpha f\circ\gamma=\chi(\gamma)f
\]
holds for some \emph{constant} $\chi(\gamma)$
of modulus $1$. All that remains is to show that $\chi$ is a 
$(\Gamma,q)$-root character. To this end, we pick two elements 
$\gamma,\tilde\gamma\in\Gamma$, and observe that 
\[
\chi^q(\tilde\gamma\circ\gamma)f^q=((\tilde\gamma\circ\gamma)')^p 
f^q\circ\tilde\gamma\circ\gamma=(\gamma')^p((\tilde\gamma')^p\circ\gamma)
\,f^q\circ\tilde\gamma\circ\gamma=
\chi^q(\gamma)(\tilde\gamma')^p f^q\circ\tilde\gamma=
\chi^q(\gamma)\chi^q(\tilde\gamma)f^q,
\]
which shows that 
$\chi^q(\tilde\gamma\circ\gamma)=\chi^q(\gamma)\chi^q(\tilde\gamma)$
and hence that $\chi$ is a $(\Gamma,q)$-root character.
\end{proof}

A function $f$ which meets condition (b) of Proposition \ref{prop-periodic1.2}
is said to be $q$-\emph{root character-periodic (or modular) of weight 
$\alpha$}, with respect to the character $\chi:\Gamma\to\Te$. Here, we could 
mention that when $\alpha=1$, they are character-periodic of weight $1$, and
called \emph{Prym differentials}. 

It is a natural question when there exist nontrivial functions $f$ with 
property (a) of Proposition \ref{prop-periodic1.2}, for a given real number 
$\alpha$. For instance, if this does not happen for irrational $\alpha$, 
then Proposition \ref{prop-periodic1.2} gives a rather complete picture. 
To sort this matter out, we consult Proposition \ref{prop-periodic1.1}, 
which asserts that the associated function $F(z):=(1-|z|^2)^\alpha|f(z)|\ge0$ 
is $\Gamma$-periodic, and note that the (real-valued) logarithm
\begin{equation}
\log F(z)=\alpha\log(1-|z|^2)+\log|f(z)|
\label{eq-logrel1.01}
\end{equation}
is $\Gamma$-periodic as well.
We now turn to 
% It turns out to be a natural point of departure to begin with 
the logarithmic monopole $U(z,w)$ for the surface $\mathcal{S}\cong\D/\Gamma$,
which has 
\begin{equation}
(\hDelta^{\mathcal{S}} U(\cdot,w))\diff A_{\mathcal{S}}=(\hDelta U(\cdot,w))
\diff A
=-\frac{1}{2a(\Gamma)}\diff A_{\mathbb{H}}
+\frac{1}{2}\sum_{\gamma\in\Gamma}\delta_{\gamma(w)}
\label{eq-monopole2}
\end{equation}
in the sense of distribution theory, where $\delta_{\zeta}$ denotes is the
unit point mass at $\zeta$, considered as a $2$-form. 
%Here, $a(\Gamma)$ stands for the $\diff A_{\mathbb{H}}$-area of the 
%fundamental polygon $D_\Gamma$. 
Note that on the right-hand side of \eqref{eq-monopole2}, we have a 
$\frac12$ point mass per tile (here, a tile is the image of $D_\Gamma$ under an 
element of $\Gamma$), which is perfectly compensated on each tile by the 
hyperbolically uniform measure $-\frac{1}{2a(\Gamma)}\diff A_{\mathbb H}$. 
As we apply the Laplacian to the relation \eqref{eq-logrel1.01}, we find that
\begin{equation}
(\hDelta\log F)\diff A=-\alpha\diff A_{\mathbb{H}}+
\frac12\sum_{\zeta\in\mathrm{Z}(f)}\delta_\zeta,
\label{eq-logrel1.02}
\end{equation}
where $\mathrm{Z}(f)$ denotes the zeros of $f$, counting multiplicities.
Note that since $F$ was $\Gamma$-periodic, the zero set $\mathrm{Z}(f)$
is $\Gamma$-periodic as well. As the surface $\mathcal{S}$ was compact, 
$F$ (and equivalently $f$) can have only finitely many zeros in 
$\D/\Gamma$, say $\zeta_1,\ldots,\zeta_n\in \bar D_\Gamma$, where some of 
the points are allowed to be on the boundary of the fundamental polygon.  
We now form the function $V_F$, 
\[
V_F(z):=\sum_{j=1}^{n}U(z,\zeta_j),
\] 
which is $\Gamma$-periodic and has Laplacian
\[
(\hDelta V_F)\diff A=\sum_{j=1}^{n}(\hDelta U(\cdot,\zeta_j))\diff A
=-\frac{n}{2a(\Gamma)}\diff A_{\mathbb{H}}+\frac12\sum_{j=1}^{n}
\sum_{\gamma\in\Gamma}\delta_{\gamma(\zeta_j)},
\]
in the sense of distribution theory. Moreover, since the points 
$\gamma(\zeta_j)$, with $j=1,\ldots,n$ and $\gamma\in\Gamma$, run through
the zero set $\mathrm{Z}(f)$, the above relation simplifies to 
\begin{equation}
(\hDelta V_F)\diff A
=-\frac{n}{2a(\Gamma)}\diff A_{\mathbb{H}}+\frac12\sum_{\zeta\in\mathrm{Z}(f)}
\delta_{\zeta},
\label{eq-logrel1.03}
\end{equation}
which we may compare with \eqref{eq-logrel1.02}. 
The difference of $V_F$ and $\log F$ is $\Gamma$-periodic, with Laplacian
\begin{equation}
(\hDelta (V_F-\log F))\diff A(z)
=\bigg(\alpha-\frac{n}{2a(\Gamma)}\bigg)\diff A_{\mathbb{H}},
\label{eq-logrel1.04}
\end{equation}
which expression has constant sign.
Since a subharmonic function on a compact Riemann surface must be constant,
we conclude that this is only possible if $V_F-\log F$ is constant and hence
$\alpha=\frac{n}{2a(\Gamma)}$. Finally, an application of the Gauss-Bonnet 
theorem gives that $a(\Gamma)=g-1$, where $g\ge2$ is the genus of the surface
$\mathcal{S}\cong\D/\Gamma$, so that in particular $\alpha=\frac{n}{2(g-1)}$
is rational. 
We gather these simple observations in a proposition.

\begin{prop}
Let $\mathcal{S}\cong\D/\Gamma$ be a compact Riemann surface of genus 
$g\ge2$. Then the $\diff A_{\mathbb{H}}$-area of the fundamental polygon $D_\Gamma$
equals $a(\Gamma)=g-1$.
Suppose $f:\D\to\C$ is holomorphic and such that the associated
function $F(z):=(1-|z|^2)^\alpha|f(z)|$ is $\Gamma$-periodic for some real
parameter $\alpha$. If $f$ is nontrivial, then $n:=2a(\Gamma)\alpha\,$ is a
nonnegative integer, and $F$ takes the form 
\begin{equation}
F=C\exp\bigg(\sum_{j=1}^{n} U(\cdot,\zeta_j)\bigg),
\label{eq-Frep1}
\end{equation}
where $U(z,w)$ is the logarithmic monopole for $\mathcal{S}$, $C$ is a positive
constant, and $\zeta_1,\ldots,\zeta_n$ enumerate the zeros of $F$ in 
$\D/\Gamma$. Moreover, any function of the form \eqref{eq-Frep1} can be written
as $F(z)=(1-|z|^2)^\alpha|f(z)|$ for some holomorphic function $f$ on $\D$
if $\alpha=\frac{n}{2a(\Gamma)}$.  
\label{prop-Fcharact1.01}
\end{prop}

\begin{proof}
All the assertions are settled by the arguments preceding the statement of the
proposition, except that it remains to show that a function $F$ given by 
\eqref{eq-Frep1} can be written as $F(z)=(1-|z|^2)^\alpha|f(z)|$ with 
$\alpha=\frac{n}{2a(\Gamma)}$, for some holomorphic $f$. 
We see from \eqref{eq-Frep1} that the purported $f$ should have
$|f(z)|=(1-|z|^2)^{-\alpha}F(z)$ and hence
\[
\log|f(z)|=-\alpha\log(1-|z|^2)+\log F(z)=-\alpha\log(1-|z|^2)+\log C
+\sum_{j=1}^{n} U(z,\zeta_j).
\]
Taking Laplacians on both sides we get that
\[
\hDelta\log|f(z)|=\frac{\alpha}{(1-|z|^2)^2}+\sum_{j=1}^{n} \hDelta_z 
U(z,\zeta_j)=\frac{\alpha-\frac{n}{2a(\Gamma)}}{(1-|z|^2)^2}
+\frac12\sum_{j=1}^{n}\sum_{\gamma\in\Gamma} \delta_{\gamma(\zeta_j)}=
\frac12\sum_{j=1}^{n}\sum_{\gamma\in\Gamma} \delta_{\gamma(\zeta_j)},
\]
in the sense of distribution theory, where we use that 
$\alpha=\frac{n}{2a(\Gamma)}$. This just asks for $f$ to have zeros 
(counting multiplicities) along the sequence of points $\gamma(\zeta_j)$, 
with $\gamma\in\Gamma$ and $j=1,\ldots,n$. A version of the Weierstrass 
factorization theorem (see, e.g., \cite{Rudbook}) assures us that there 
exists a holomorphic function $h:\D\to\C$ with precisely the zeros prescribed
for $f$, and then the difference $u:=\log|f|-\log|h|$ must be harmonic.
By forming the harmonic conjugate to $u$ we obtain a holomorphic function
$U:\D\to\C$ with real part equal to $u$. Finally, we realize that the choice
$f:=\e^U h$ is holomorphic with the right modulus so that 
$F(z)=(1-|z|^2)|f(z)|$ holds. 
\end{proof}

\begin{rem}
For some related geometric complex analysis on compact Riemann surfaces, 
involving forms and sections, see the textbooks \cite{Bakerbook}, 
\cite{Fay}, \cite{Hej}.
\end{rem}

\subsection{Character-modular forms, ergodic geodesic flow,
and hyperbolic zero packing}

We keep the setting of a compact Riemann surface $\mathcal{S}$ 
with genus $\ge2$, so that $\mathcal{S}\cong\D/\Gamma$ for a Fuchsian group
with a fundamental domain $D_\Gamma$. Then, as a matter of definition (see 
\eqref{eq-rhosurf}),
\[
\rho_{n,\beta}(\mathcal{S})=
\inf_{b,w_1,\ldots,w_n}\big\langle
\big(b\,\e^{\beta U^{\langle n\rangle}}-1\big)^2\big\rangle_{\mathcal{S}},
\]
where $b$ is a positive real and $w_1,\ldots w_n\in\mathcal{S}$. 
We would like to see how this fares compared with the hyperbolic densities 
$\rho_{\alpha,\beta}(\mathbb{H})$ defined in Subsection \ref{hypzero-alphabeta}. 
We use Proposition \ref{prop-Fcharact1.01} to see that there exists a 
holomorphic function $f$ with zeros (counting multiplicities) exactly at 
the points $\gamma(w_j)$ when $j=1,\ldots,n$ and $\gamma\in\Gamma$, such that 
\[
b\,\e^{\beta U^{\langle n\rangle}(z)}=(1-|z|^2)^{\frac{n\beta}{2a(\Gamma)}}
|f(z)|^\beta,
\]
and both sides express $\Gamma$-periodic functions. Now, for the right-hand 
side we could try to compute the discrepancy density with respect to the 
disk $\D$ as well:
\[
\rho_{\alpha',\beta}(f):=
\liminf_{r\to1^-}\frac{1}{\log\frac{1}{1-r^2}}\int_{\D(0,r)}
\big((1-|z|^2)^{\alpha'}|f(z)|^\beta-1\big)^2\frac{\diff A(z)}{1-|z|^2},
\]
with $\alpha':=\frac{n\beta}{2a(\Gamma)}$. The following result tells us 
that the above average can be achieved by integrating over one tile with 
respect to the Fuchsian group $\Gamma$. 

\begin{prop}
In the above setting, with $f:\D\to\C$ holomorphic and the associated function
$z\mapsto(1-|z|^2)^{\alpha'}|f(z)|^\beta$ assumed $\Gamma$-periodic, we have that
\[
\lim_{r\to1^-}\frac{1}{\log\frac{1}{1-r^2}}\int_{\D(0,r)}
\big((1-|z|^2)^{\alpha'}|f(z)|^\beta-1\big)^2\frac{\diff A(z)}{1-|z|^2}
=\frac{1}{a({\Gamma})}\int_{D_\Gamma}
\big((1-|z|^2)^{\alpha'}|f(z)|^\beta-1\big)^2\diff A_{\mathbb{H}}.
\]
\end{prop}

\begin{proof}
To simplify the notation, we write
\[
\Phi_{\alpha',\beta,f}(z):=\big((1-|z|^2)^{\alpha'}|f(z)|^\beta-1\big)^2,
\]
which is $\Gamma$-periodic and hence a well-defined function on 
$\mathcal{S}\cong\D/\Gamma$. The claim can now be expressed in the form
\begin{equation}
\lim_{r\to1^-}\frac{\int_{\D(0,r)}\Phi_{\alpha',\beta,f}(z)\frac{\diff A(z)}{1-|z|^2}}
{\log\frac{1}{1-r^2}}=\frac{\int_{\mathcal{D}_\Gamma}\Phi_{\alpha',\beta,f}(z)
\diff A_{\mathbb{H}}(z)}{\int_{\mathcal{D}_\Gamma}\diff A_{\mathbb{H}}},
\label{eq-hypaverage2.1}
\end{equation}
that is, averages formed in two different ways coincide. Such assertions
remind us of ergodic theory. Indeed, it follows from the well-known 
ergodicity of geodesic flow on compact hyperbolic surfaces, originally 
due to Hopf and later extended to higher-dimensional manifolds by Anosov 
(see, e.g., Hopf's expository paper \cite{Hopf}). 
The ``time average'' of $\Phi_f$ over the geodesic ray $z=t\zeta$ with 
$|\zeta|=1$ and radial parameter $t$ with $0<t<r$, is
\begin{equation}
\frac{\int_0^r\Phi_f(t\zeta)\frac{\diff t}{1-t^2}}
{\int_0^r\frac{\diff t}{1-t^2}}=
\frac{\int_0^r\Phi_f(t\zeta)\frac{\diff t}{1-t^2}}
{\frac12\log\frac{1+r}{1-r}}, 
\label{eq-hypaverage2.2}
\end{equation}
while the ``space average'' is expressed by the right-hand side of
\eqref{eq-hypaverage2.1}. The limit of the ratio \eqref{eq-hypaverage2.2}
as $r\to1^-$ is clearly unperturbed if we replace $\Phi_f(t\zeta)$ by
$t\Phi_f(t\zeta)$, and hence \eqref{eq-hypaverage2.1} results from 
\eqref{eq-hypaverage2.2} by integration over the circle $\Te$ in $\zeta$.
Here, we used the elementary observation that
\[
\lim_{r\to1^-}\frac{\log\frac{1}{1-r^2}}{\log\frac{1+r}{1-r}}=1.
\]
The proof is complete.
\end{proof}

\begin{rem}
Using more refined control of the error term in the ergodicity of geodesic
flow, it is possible to obtain a comparison of the densities 
$\rho_{n,\beta}(\mathcal{S})$ and $\rho_{\alpha',\beta}(\mathbb{H})$, to the 
effect that
\[
\rho_{\alpha',\beta}(\mathbb{H})\le\rho_{n,\beta}(\mathcal{S})\quad \text{where}
\quad\alpha'=\frac{n\beta}{2a(\Gamma)}\quad\text{and}\quad 
\mathcal{S}\cong\D/\Gamma.
\]
\label{rem-final}
The question comes to mind if, for fixed $\alpha'$ and $\beta$, the 
right-hand side expression can be made arbitrarily close to the left-hand 
side expression by varying suitably the surface $\mathcal{S}$ and the 
number of points $n$. This may not be the case. 
\end{rem}

\section{Fekete configurations, geometric zero packing,
and heat flow on compact Riemann surfaces}
\label{sec-Feketegeomzero}

This section takes the form of a rather extensive remark on heat flows
and the relation between Fekete-type problems and geometric zero packing on a
compact Riemann surface $\mathcal{S}$ supplied with a metric of constant 
Gaussian curvature. The particular normalizations and notational conventions
are kept as before.

\subsection{Integration along heat flow}
We begin with the comparison of Remark \ref{rem-energies}, in the context of
heat flow on a compact Riemann surface $\mathcal{S}$. 
First, let 
\begin{equation}
U^{\langle n\rangle}(z)=U(z,w_1)+\cdots+U(z,w_n),\qquad z\in\mathcal{S},
\label{eq-fun:u0}
\end{equation}
be a sum of logarithmic monopoles, each normalized so that 
\eqref{eq-rhosurf0.1} holds. For $t>0$, we let $U^{\langle n\rangle}_t(z)$ denote
the function of $(z,t)$ for $z\in\mathcal{S}$ and $t>0$ which solves 
the \emph{heat equation}
\begin{equation}
\partial_t U^{\langle n\rangle}_t(z)=\hDelta^{\mathcal{S}}U^{\langle n\rangle}_t(z),
\qquad z\in\mathcal{S},
\label{eq-heateq}
\end{equation}
with initial datum at $t=0$ given by $U^{\langle n\rangle}(z)$. 
Note that in view of 
\eqref{eq-rhosurf0.1}, it follows that $U^{\langle n\rangle}_t\to0$ as 
$t\to+\infty$.
Then according to Remark \ref{rem-energies} the Fekete configuration
problem is concerned with minimizing the Dirichlet energy
\begin{equation}
\int_{\mathcal{S}}
|\nabla^{\mathcal{S}} U^{\langle n\rangle}_t|^2\diff A_{\mathcal{S}}
\label{eq-Direnergy}
\end{equation}
in the limit as $t\to0^+$ over all point configurations $w_1,\ldots,w_n$, 
where the gradient and area elements are geometrically adjusted. 
On the other hand, the $\beta\to0$ geometric zero
packing problem \eqref{eq-rhosurf0.0} minimizes instead the Bergman energy
\begin{equation}
\int_{\mathcal{S}}
(U^{\langle n\rangle}_t)^2\diff A_{\mathcal{S}}
\label{eq-Bergenergy}
\end{equation}
for $t=0$.
We would like to relate, if possible, these two problems. Suppose we are lucky
and the minimizing configuration for the Dirichlet energy \eqref{eq-Direnergy} 
as $t\to0^+$ \emph{actually minimizes for all $t>0$ at the same time}.
Then \emph{the same configuration is also minimizing for the Bergman energy 
as well}.
To see this, we calculate as follows:
\begin{multline}
\int_0^{+\infty}\int_{\mathcal{S}}
|\nabla^{\mathcal{S}}U^{\langle n\rangle}_t|^2\diff A_{\mathcal{S}}\diff t
=-4\int_0^{+\infty}
\int_{\mathcal{S}}U^{\langle n\rangle}_t\hDelta^{\mathcal{S}}U^{\langle n\rangle}_t
\diff A_{\mathcal{S}}\diff t
\\
=-4\int_0^{+\infty}
\int_{\mathcal{S}}U^{\langle n\rangle}_t
\partial_t U^{\langle n\rangle}_t\diff A_{\mathcal{S}}\diff t=
-2\int_{\mathcal{S}}\int_0^{+\infty}
\partial_t\{(U^{\langle n\rangle}_t)^2\}\diff t\diff A_{\mathcal{S}}
=2\int_{\mathcal{S}}(U^{\langle n\rangle})^2\diff A_{\mathcal{S}}.
\label{eq-DirBergenergies}
\end{multline}
In any case, the formula \eqref{eq-DirBergenergies} shows that the Bergman 
energy is the time integral of the Dirichlet energy along the heat flow.
We now search for an analogue of \eqref{eq-DirBergenergies} which applies for
more general values of the parameter $\beta$.

\subsection{Integration along $\beta$-deformed heat flow}
Let $V^{\langle n\rangle}_{b,\beta}$ denote the function
\[
V^{\langle n\rangle}_{b,\beta}:=\frac{b\e^{\beta U^{\langle n\rangle}}-1}{\beta},
\]
where $b$ is positive and real and $U^{\langle n\rangle}$ is as in 
\eqref{eq-fun:u0}; the minimum
\[
\inf_b\langle(V^{\langle n\rangle}_{b,\beta})^2\rangle_{\mathcal{S}},
\]
which is of interest in connection with the problem of geometric $\beta$-zero
packing, is then attained for
\[
b=b(\beta):=
\frac{\langle\e^{\beta U^{\langle n\rangle}}\rangle_{\mathcal{S}}}
{\langle\e^{2\beta U^{\langle n\rangle}}\rangle_{\mathcal{S}}}.
\] 
To simplify the notation, we write
$V^{\langle n\rangle}_\beta:=V^{\langle n\rangle}_{b(\beta),\beta}$. After all, 
the geometric zero packing problem involves the same minimization over 
all constants $b$ and all point configurations $w_1,\ldots,w_n$ 
(which determine the function $U^{\langle n\rangle}$).  
We then observe that $V^{\langle n\rangle}_{\beta}\to U^{\langle n\rangle}$ as 
$\beta\to0$, which suggests that we may think of $V^{\langle n\rangle}_\beta$ 
as a (nonlinear) $\beta$-deformation of the function 
$U^{\langle n\rangle}$. The average of $V^{\langle n\rangle}_\beta$ equals
\[
\langle V^{\langle n\rangle}_\beta\rangle_{\mathcal{S}}=
\frac{1}{\beta}\Bigg(\frac{\langle\e^{\beta U^{\langle n\rangle}}
\rangle^2_{\mathcal{S}}}
{\langle\e^{2\beta U^{\langle n\rangle}}\rangle_{\mathcal{S}}}-1
\Bigg)
\]
which vanishes only when $\beta=0$. This means that for $\beta\ne0$, 
after infinite time the heat limit with initial datum $V^{\langle n\rangle}_\beta$ 
equals a nonzero constant, which makes the approach underlying the formula 
\eqref{eq-DirBergenergies} difficult to carry out. 
However, it is indeed the case that a weighted average of 
$V^{\langle n\rangle}_\beta$ vanishes: 
\[
\langle V^{\langle n\rangle}_\beta\e^{\beta U^{\langle n\rangle}}\rangle_{\mathcal{S}}=0.
\]
To use this fact, we equip the surface $\mathcal{S}$ with the 
$\beta$-deformed average 
\[
\langle f\rangle_{\beta,\mathcal{S}}:=
\langle f\e^{2\beta U^{\langle n\rangle}}\rangle_{\mathcal{S}}
\]
which we understand as a deformation of the area element, and correspondingly 
we may consider the deformed geometric Laplacian 
\[
\hDelta^{\beta,\mathcal{S}}:=\e^{-2\beta U^{\langle n\rangle}}\hDelta^{\mathcal{S}}.
\]
Associated with this $\beta$-deformed Laplacian we have a $\beta$-deformed
heat flow, which we express in terms of the operator 
$\Hop_{\beta,t}=\e^{t\hDelta^{\beta,\mathcal{S}}}$. Here, to be more precise,
$v_t=\Hop_{\beta,t} v_0$ means that $v_t$ solves the geometrically deformed 
heat equation 
\[
\partial_t v_t=\hDelta^{\beta,\mathcal{S}}v_t
\]
with initial datum $v_0$ at $t=0$. Let $\tilde V^{\langle n\rangle}_\beta
:=\e^{-\beta U^{\langle n\rangle}}V^{\langle n\rangle}_\beta$,
and observe that
\[
\lim_{t\to+\infty}
\Hop_{\beta,t}\tilde V^{\langle n\rangle}_\beta\to
\langle \tilde V^{\langle n\rangle}_\beta\rangle_{\beta,\mathcal{S}}
=\langle \e^{\beta U^{\langle n\rangle}} V^{\langle n\rangle}_\beta\rangle_{\mathcal{S}}=0,
\] 
which allows us to obtain a $\beta$-deformed analogue of the identity 
\eqref{eq-DirBergenergies}:
\begin{multline}
\frac{1}{2}\int_0^{+\infty}
\big\langle|\nabla^{\mathcal{S}}\Hop_{\beta,t} \tilde V^{\langle n\rangle}_\beta|^2
\big\rangle_{\mathcal{S}}\diff t=
-2\int_0^{+\infty}
\big\langle(\Hop_{\beta,t}\tilde V^{\langle n\rangle}_\beta)
\hDelta^{\mathcal{S}}\Hop_{\beta,t} \tilde V^{\langle n\rangle}_\beta)
\big\rangle_{\mathcal{S}}\diff t
\\
=-2\int_0^{+\infty}
\big\langle(\Hop_{\beta,t}\tilde V^{\langle n\rangle}_\beta)
\hDelta^{\beta,\mathcal{S}}\Hop_{\beta,t}\tilde V^{\langle n\rangle}_\beta)
\big\rangle_{\beta,\mathcal{S}}\diff t
=-2\int_0^{+\infty}
\big\langle(\Hop_{\beta,t}\tilde V^{\langle n\rangle}_\beta)
\partial_t\Hop_{\beta,t}\tilde V^{\langle n\rangle}_\beta\big\rangle_{\beta,\mathcal{S}}
\diff t
\\
=-\int_0^{+\infty}
\big\langle\partial_t\{(\Hop_{\beta,t}\tilde V^{\langle n\rangle}_\beta)^2\}
\big\rangle_{\beta,\mathcal{S}}\diff t
=\langle(\tilde V^{\langle n\rangle}_\beta)^2\rangle_{\beta,\mathcal{S}}
=\langle(V^{\langle n\rangle}_\beta)^2\rangle_{\mathcal{S}}.
\label{eq-DirBergenergies2}
\end{multline}
Again, if for each $t>0$, the Dirichlet energy
\begin{equation}
\big\langle|\nabla^{\mathcal{S}}\Hop_{\beta,t} \tilde V^{\langle n\rangle}_\beta|^2
\big\rangle_{\mathcal{S}} 
\label{eq-Direnergybeta}
\end{equation}
is minimized for one and the same function $U^{\langle n\rangle}$ 
of the form \eqref{eq-fun:u0}, then the same goes for the Bergman norm 
on the right-hand side. This would suggest the introduction of a 
$\beta$-deformed Fekete problem, which asks for
the minimization of \eqref{eq-Direnergybeta} in the limit as $t\to0^+$ (after 
renormalization to remove the infinities which arise from the singularities at 
$w_1,\ldots,w_n$).
As a side remark, if we recall that
\[
\tilde V^{\langle n\rangle}_\beta=\e^{-\beta U^{\langle n\rangle}}V^{\langle n\rangle}_\beta
=\frac{b-\e^{-\beta U^{\langle n\rangle}}}{\beta},
\] 
and note that constants are preserved under weighted heat evolution, 
it might be reasonable to try to express the weighted heat evolution in
the form
\[
\Hop_{\beta,t}\tilde V^{\langle n\rangle}_\beta(z)=
\frac{b-\e^{-\beta \Theta(z,t)}}{\beta}.
\]
The function $\Theta(z,t)$ then evolves according to the nonlinear 
heat equation
\[
\partial_t \Theta=\e^{-2\beta U^{\langle n\rangle}}
\bigg(\hDelta^{\mathcal{S}} \Theta
-\frac{\beta}{4}
|\nabla^{\mathcal{S}}\Theta|^2\bigg),
\]
with initial datum $\Theta(z,0)=U^{\langle n\rangle}(z)$ for $t=0$.
The above equation has at least some superficial similarity with the KPZ 
(Kardar-Parisi-Zhang) equation if we let the configuration 
$\{w_1,\ldots,w_n\}$ (which determines $U^{\langle n\rangle}$ by \eqref{eq-fun:u0}) 
be stochastic and perhaps time-dependent. 
For the KPZ equation, see e.g. the survey paper \cite{Cor} and the original 
contribution by Kardar, Parisi, and Zhang \cite{KPZ}.

\end{document}